\documentclass[11pt]{amsart}

\usepackage{mathpazo}
\usepackage[all,cmtip]{xy}
\usepackage{amsmath, amssymb, enumitem, mathrsfs}
\usepackage{mathtools}
\usepackage[T1]{fontenc}
\usepackage[english]{babel}
\usepackage[hidelinks]{hyperref}
\usepackage[margin=1in]{geometry}
\usepackage{xcolor}\usepackage{marginnote}
\usepackage{csquotes}
\usepackage{dsfont}
\numberwithin{equation}{subsection}

\newtheorem{thm}{Theorem}[subsection]

\newtheorem*{thm*}{Theorem}
\newtheorem{lem}[thm]{Lemma}

\newtheorem{prop}[thm]{Proposition}
\newtheorem*{prop*}{Proposition}

\newtheorem{cor}[thm]{Corollary}

\theoremstyle{definition}
\newtheorem{defn}[thm]{Definition}
\newtheorem{notation}[thm]{Notation}
\newtheorem{convention}[thm]{Convention}
\newtheorem{remark}[thm]{Remark}

\hypersetup{
    colorlinks=true, %set true if you want colored links
    linktoc=all,     %set to all if you want both sections and subsections linked
    citecolor = red,
    linkcolor=blue,  %choose some color if you want links to stand out
}

\DeclareFontFamily{U}{mathx}{}
\DeclareFontShape{U}{mathx}{m}{n}{<-> mathx10}{}
\DeclareSymbolFont{mathx}{U}{mathx}{m}{n}
\DeclareMathAccent{\widehat}{0}{mathx}{"70}
\DeclareMathAccent{\widecheck}{0}{mathx}{"71}

\def\eqdef{\stackrel{\text{def}}{=}}

\def\cal{\mathcal}
\def\frak{\mathfrak}

\def\bf{\mathbf}

\DeclareMathOperator{\id}{id}
\DeclareMathOperator{\ad}{ad}

\DeclareMathOperator{\Res}{Res}

\DeclareMathOperator{\End}{End}
\DeclareMathOperator{\Span}{span}

\DeclareMathOperator{\Spec}{Spec}

\DeclareMathOperator{\Conf}{Conf}
\DeclareMathOperator{\Rav}{Rav}

\newcommand{\normord}[1]{%
  :\mathrel{\mspace{2mu}#1\mspace{2mu}}:%
}

\def\R{\mathbb R}

\def\Z{\mathbb Z}
\def\Q{\mathbb Q}
\def\C{\mathbb C}
\def\N{\mathbb N}

\def\A{\mathbb A}

\begin{document}

%%%%%%%%%%%%%%%%%%%%%%%%%%%%%%%%%%%%%%%%%%%%%%
\title{Cohomological Vertex Algebras}

\author[Griffin]{Colton Griffin}

\address{David Rittenhouse Lab, University of Pennsylvania, 209 South 33rd Street, Philadelphia, PA 19104-6395}

%%%%%%%%%%%%%%%%%%%%%%%%%%%%%%%%%%%%%%%%%%%%%%

\setcounter{tocdepth}{1}

\begin{abstract}
    Vertex algebras (and their modules) can be described as vector spaces together with a linear operator-valued series in one parameter $z$. With the interpretation of $z$ as a coordinate at a point on a curve, one can construct algebraic structures on the moduli space of curves from $V$-modules. Here we propose a generalization of vertex algebras involving linear operators in parameters $z_1,\ldots,z_n$. One may interpret these as being the components of a set of coordinates on an $n$-dimensional algebraic variety. These are referred to as cohomological vertex algebras (CVAs): the formal punctured 1-disk underlying a vertex algebra is replaced by a ring modeling the cohomology of certain modifications of the formal $n$-disk. We prove several structural theorems for CVAs and give a definition of cohomological vertex operator algebras (CVOAs). Using a reconstruction theorem for CVAs, we provide basic examples such as the $\beta\gamma$-system, the Heisenberg CVA, and the affine Kac-Moody CVAs. We use these constructions to describe BRST reduction, leading to an analog of W-algebras.
\end{abstract}

\maketitle

% \tableofcontents

\section{Introduction}
Vertex algebras (VAs) are fundamental algebraic structures that arise naturally in a broad context. Among their many applications, VAs provide a framework for constructing algebraic structures on moduli spaces: Given a pointed algebraic curve, we may associate to the distinguished point a VA whose formal parameter $z$ is a choice of local coordinate on the curve. If the VA is a vertex operator algebra (VOA), we may vary this choice of formal parameter and the point on the curve, giving rise to sheaves of coinvariants on moduli spaces of stable curves
\cite{frenkel2004vertex}, \cite{DGT21}, \cite{DGT-factorization}. 
This connection between VOAs and moduli spaces motivates a natural question: \textit{do these notions extend to higher dimensions?} More specifically, are there ``higher-dimensional'' analogs of VAs that are suitable for the application of building analogous structures on moduli spaces of higher-dimensional varieties? One may also hope that such a generalization will also be useful in the study of higher-dimensional field theories in physics.

Generalizing VAs is a subtle problem, and there are several perspectives one could take. For example, it is known that VAs are equivalent to translation-equivariant chiral algebras on the affine line $\A^1$ \cite{BD2004chiral}, so one could define an \textit{$n$-dimensional VA} to be a translation-equivariant chiral algebra on $\A^n$ in the sense of Francis \& Gaitsgory \cite{FrancisGaitsgory}. With chiral algebras, the use of formal calculus in VAs is replaced with $\cal D$-modules and other tools in derived algebraic geometry. This approach to higher-dimensional VAs is taken in the study of the unit chiral algebra on $\A^n$ in \cite{YoungChiral}.  However, this definition can be difficult to work with due to the fact that translation-equivariant chiral algebras on $\A^n$ for $n\in \Z_{\geq 2}$ define an $\infty$-category rather than a usual 1-category, making it less straightforward to translate the usual algebraic VA axioms to higher dimensions.

In this paper, we introduce \textit{cohomological vertex algebras} (CVAs), which are based on a simple idea. Recall that classical VAs can be understood through the ring of Laurent series $\C(\!(z)\!)$, which is naturally identified with the global sections $H^0(D_1^\circ,\cal O)$ of the formal punctured disk $D_1^\circ = \Spec\C(\!(z)\!)$. To generalize to higher dimensions, we replace $D_1^\circ$ with any of the following schemes $D$:
\begin{enumerate}
    \item The formal punctured $n$-disk $D = D_n^\circ = \Spec\C[\![z_1,\ldots,z_n]\!]\setminus\{(z_1,\ldots,z_n)\}$ for $n\in \Z_{\geq 2}$.
    \item $D = D_n^{(p)}$, the formal $n$-disk with a $p$-multiplied origin for $p\in \Z_{\geq 2}$ and $n\in \Z_{\geq 1}$. Here $D_n^{(p)}$ is given as the colimit of the diagram whose objects are $D_n^\circ$ and $p$ copies of $D_n$, denoted $U_i\cong D_n$ for all $i=1,\ldots,p$, and the morphisms are the canonical maps $D_n^\circ \to U_i$. When $n=1$ and $p=2$, then we get the formal raviolo $\Rav = D_1^{(2)} = D_1\sqcup_{D_1^\circ}D_1$.
\end{enumerate}
For either of the above choices for $D$, the \v{C}ech cohomology ring $H^\bullet(D,\cal O)$ is concentrated in two degrees. In degree 0, we obtain the usual ring of regular functions, and in some higher degree $N$ we obtain a collection of purely singular functions:
\[H^\bullet(D,\cal O) = \begin{dcases}
\C[\![z_1,\ldots,z_n]\!],& \bullet = 0\\
z_1^{-1}\ldots z_n^{-1}\C[z_1^{-1},\ldots ,z_n^{-1}],& \bullet = N
\end{dcases}\]
Here $N=n-1$ for $D = D_n^\circ$ and $N = n+p-2$ for $D = D_n^{(p)}$.
We give a natural set of replacement axioms for VAs in this setting, referring to the resulting algebraic objects as \textit{cohomological vertex algebras} (CVAs).

The second option $D = D_n^{(p)}$ is closely related to the raviolo VAs of \cite{ravioli}, in which Garner \& Williams develop a VA-like structure with $\C(\!(z)\!)$ replaced by the $\C$-algebra $H^\bullet(\Rav,\cal O)$, where $\Rav \eqdef D_1^{(2)} = D_1\sqcup_{D_1^\circ}D_1$ is the formal raviolo (or 1-disk with a doubled origin). We adopt much of the paper's notation, and we adapt and extend many of their results.

\subsection{Our results}
Cohomological vertex algebras satisfy several properties in common with vertex algebras. For example, we show there are natural analogs of locality/weak commutativity, a weak version of the Dong-Li lemma (Lemma \ref{lem:Dongpoly}), skew-symmetry (Proposition \ref{prop:skewsymmetry}), Goddard uniqueness (Proposition \ref{prop:Goddard}), weak associativity (Proposition \ref{prop:weakassoc}), and the Jacobi identity (Theorem \ref{thm:Jacobi}), a foundational identity in the theory of vertex algebras. We also examine some equivalent axioms for CVAs. In particular, we show that the axioms of a CVA may be reformulated in terms of generalized correlation functions (Theorem \ref{thm:CVAcorrfuns}). This is related to the correspondence between chiral algebras and factorization algebras \cite{BD2004chiral}.

The most widely studied VAs are vertex operator algebras over $\C$ (VOAs), which have a compatible action by the Virasoro algebra. For CVAs, we similarly define a Lie algebra $CW_n$ similar to the Witt algebra, and it admits a central extension of dimension 1 for $n=1$ and is centrally closed when $n\geq 2$. This allows us to give the definition of a cohomological vertex operator algebra (CVOA).

Finally, using an analog of the reconstruction theorem for vertex algebras (Theorem \ref{thm:reconstruction}), we provide some examples of cohomological vertex algebras, which are natural analogs of the Virasoro, $\beta\gamma$-free field, affine Kac-Moody, and Heisenberg vertex algebras. We finish the section with a discussion on BRST reduction and W-algebras. We expect that these constructions are related to certain higher-dimensional chiral algebras on $\A^n$, generalizing the corresponding construction for vertex operator algebras.

We refer the reader to an extended version of this paper as a preprint \cite{CVA} for an extended discussion of local fields, normal ordered products, and the Lie algebra $CW_n$.

There are many other important general constructions in the theory of VAs, such as modules, vertex Lie algebras, Zhu algebras, coinvariants/conformal blocks, and intertwining operators. Indeed all of these constructions may be performed and will be explored in-depth in a follow-up paper. We hope that there is some kind of modular invariance result for CVAs, just as there is for VAs.

We conjecture that the homotopy category of translation-equivariant chiral algebras on $\A^n$ is equivalent to the category of CVAs for $D = D_n^\circ$. We will not pursue this conjecture here.
Part of the reason we are working with CVAs instead of chiral algebras is that we do not have a good description of what cohomological vertex algebras for $D = D_n^{(p)}$ should correspond to in terms of chiral-algebraic structures. We expect that there exists a corresponding chiral-algebraic analog depending on $p$ and $n$, but it is unclear what modifications must be made to the definition from Francis \& Gaitsgory. See \cite{Tam02} for an example of such objects on curves.

\subsection{Plan of the paper}
In Section \ref{Sec2} we give an overview of the relevant formal calculus for the definition of CVAs. 

In Section \ref{Sec3}, we define CVAs and state their basic properties. We provide analogs of famous properties of classical vertex algebras such as weak commutativity, weak associativity, skew-symmetry, the iterate formula, and the Jacobi identity. We also discuss various equivalent sets of axioms, showing for example that some of our axioms may be equivalently replaced by the Jacobi identity or skew-symmetry together with the iterate formula (as is done in Borcherds' definition). We also discuss how to identify the data of a vertex algebra using correlation functions.

In Section \ref{Sec5}, we provide examples of cohomological vertex algebras for all dimensions $n\in \Z_{\geq 1}$, with our primary result being a cohomological analog of the reconstruction theorem. We provide analogous constructions of the Virasoro, $\beta\gamma$-free field, affine Kac-Moody\footnote{The primary inspiration for our work is \cite{KapranovKacMoody}, in which Faonte, Hennion, \& Kapranov examine a derived Kac-Moody algebra $\frak g^\bullet_n \eqdef \frak g\otimes_\C R\Gamma(D_n^\circ,\cal O)$ for a reductive Lie algebra $\frak g$. For $n=1$ we obtain the usual current algebra $\frak g(\!(z)\!)$, whose universal central extension gives rise to the famous affine vertex operator algebras. In Section 3.2 they describe a central extension of the dg Lie algebra $\frak g^\bullet_n$ as an $L_\infty$ algebra. This motivated us to analyze the cohomology $H^\bullet(\frak g^\bullet_n)$ and determine if a natural VA-like structure exists in the way it does for $n=1$. This motivated the definition of the affine CVA.}, and Heisenberg vertex algebras. We end the paper with a discussion of BRST reduction and W-algebras, following \cite{ravioli}.

\subsection*{Acknowledgements}
We thank Angela Gibney, Daniel Krashen, James Lepowsky, Yi-Zhi Huang, Niklas Garner, and Brian Williams for the helpful conversations that led to the development of this paper.

\section{Preliminaries on formal calculus and local systems of fields}\label{Sec2}
We start with a discussion of formal calculus in Section \ref{subsection:formalcalc} and Section \ref{subsec:taylor}, moving on to the main definition of CVAs in Section \ref{subsec:CVA}. We start with some conventions.
Throughout this section, we will fix integers $N,n\in \Z$ such that $n\in \Z_{\geq 1}$. By convention, all rings are commutative.
Given $n\in \Z_{\geq 1}$, we denote $\bf z = (z_1,\ldots,z_n)$ for an ordered tuple of formal variables. We denote an arbitrary element of $\N^n$ using bold notation, that is $\bf k = (k_1,\ldots,k_n)$. For $\bf k,\bf j\in \N^n$, we set
\[\bf {z^k} \eqdef \prod_{i=1}^n z_i^{k_i},\qquad \bf k! \eqdef \prod_{i=1}^nk_i!,\qquad \binom{\bf k}{\bf j} \eqdef \prod_{i=1}^n\binom{k_i}{j_i}.\]
We write $\bf k\geq \bf j$ to mean $k_i\geq j_i$ for all $i=1,\ldots,n$. We also write $\bf n = (n,\ldots,n)$ for $n\in \Z$.

\subsection{Formal Calculus}\label{subsection:formalcalc}
Here, we develop the basics of the formal calculus that is necessary for the rest of the work done in this paper. To start, we give some conventions for $\Z$-graded $R$-modules, where $R$ is a commutative ring.
\begin{convention}
    Let $V = \bigoplus_{r\in \Z}V^r$ be a $\Z$-graded $R$-module. Given a homogeneous element $a\in V$, we denote its degree as $p(a)$. Given two homogeneous elements $a,b\in V$, we write
    \[p(a,b) = (-1)^{p(a)p(b)}.\]
    Later on, we will use $|k|$ to mean the sum of elements of a multi-index $k\in \N^n$, which is why we do not use this notation here.
\end{convention}
Fix a ring $R$, a $\Z$-graded $R$-module $V = \bigoplus_{r\in \Z}V^r$, and integers $n,N$ such that $n\geq 1$. We define $\cal K_{\mathrm{dist}}^{\bf z}(V)$ to be the graded $R$-module of series of the form
\begin{equation}\label{eq:f(z)}
    A(\bf z) = \sum_{\bf k\in \N^n}\bf {z^k}A_{-\bf 1-\bf k} + \Omega_{\bf z}^{\bf k}A_{\bf k},
\end{equation}
where $A_{-\bf 1-\bf k},A_{\bf k}\in V$, the variables $z_i$ are degree 0, and the symbols $\Omega_{\bf z}^{\bf k}$ are all degree $N$. The $R$-action is given by the $R$-action on $V$. Here the forms $\Omega_{\bf z}^{\bf k}$ correspond to the elements $\bf z^{-\bf 1-\bf k}$ in $H^N(D,\cal O)$, which was described in the introduction.
Here $n$ and $N$ are not explicitly stated when writing $\cal K^{\bf z}_{\mathrm{dist}}(V)$, and they should be assumed to be arbitrary unless stated otherwise.

We further define two submodules:
\begin{enumerate}
    \item $\cal K^{\bf z}(V)\subset \cal K^{\bf z}_{\mathrm{dist}}(V)$ is the $R$-submodule of series $A(\bf z)$ such that there exists $\bf K\in \N^n$ such that for $\bf k\in \N^n$ we have $A_{\bf k} = 0$ unless $\bf k\leq \bf K$.
    \item $\cal K^{\bf z}_{\mathrm{poly}}(V)\subset \cal K^{\bf z}(V)$ is the $R$-submodule of series $A(\bf z)$ such that there exists $\bf K\in \N^n$ such that $A_{-\bf 1-\bf k} = 0$ and $A_{\bf k} = 0$ unless $\bf k\leq \bf K$.
\end{enumerate}
To abbreviate the above constructions, we write
\begin{align}
    \cal K_{\mathrm{dist}}^{\bf z}(V) &\eqdef V[\![\bf z]\!]\oplus V[\![\Omega_{\bf z}^{\bf k}]\!]_{\bf k\in \N^n},\\
    \cal K^{\bf z}(V) &\eqdef V[\![\bf z]\!]\oplus V[\Omega_{\bf z}^{\bf k}]_{\bf k\in \N^n},\\
    \cal K_{\mathrm{poly}}^{\bf z}(V) &\eqdef V[\bf z]\oplus V[\Omega_{\bf z}^{\bf k}]_{\bf k\in \N^n},
\end{align}

The $R$-modules
$\left(\cal K^{\bf z}_{\mathrm{poly}}(V), \cal K^{\bf z}(V),\cal K^{\bf z}_{\mathrm{dist}}(V)\right)$ are analogs of $\left(V[z^{\pm 1}],V(\!(z)\!),V[\![z^{\pm 1}]\!]\right)$, respectively. We call these the \textit{spaces of cohomological Laurent polynomials, Laurent series}, and \textit{formal distributions} respectively.
To make the analogy with Laurent series appropriate, sometimes we write $V\langle\!\langle \bf z\rangle\!\rangle\eqdef\cal K^{\bf z}(V)$.

Now, let $U$ be a $\Z$-graded commutative $R$-algebra. We impose a product structure on $\cal K^{\bf z}(U)$ as follows. For all $\bf j,\bf k\geq \bf 0$ we set
\begin{equation}
\bf z^{\bf j}\Omega_{\bf z}^{\bf k} = \Omega_{\bf z}^{\bf k}z^j = \begin{dcases}
    \Omega_{\bf z}^{\bf k-\bf j},& j\leq k\\
    0,& \text{else},
\end{dcases}\qquad \bf z^{\bf j}\bf z^{\bf k} = \bf z^{\bf j+\bf k} = \bf z^{\bf k}\bf z^{\bf j},\qquad \Omega_{\bf z}^{\bf j}\Omega_{\bf z}^{\bf k} = 0.
\end{equation}
We require that elements of $U$ graded-commute with the forms $\Omega^{\bf  k}_{\bf z}$. That is, for $a\in V^r$ we require that $a\Omega^{\bf k}_{\bf z} = (-1)^{rN}\Omega^{\bf k}_{\bf z}a$, where again $N$ is the degree of $\Omega^{\bf k}_{\bf z}$ for all $\bf k\geq 0$.
The $R$-submodule $\cal K^{\bf z}_{\mathrm{poly}}(U)$ is naturally a subalgebra of $\cal K^{\bf z}(U)$.
\begin{remark}
    The first and last products are what distinguish $\cal K^{\bf z}(R)$ from the ring of Laurent series $R(\!(z)\!)$, aside from the $\Z$-grading and that $n$ can be greater than 1.
\end{remark}
The modules $\cal K^{\bf z}(V),\cal K^{\bf z}_{\mathrm{poly}}(V)$ also come with natural degree 0 endomorphisms $\partial_{\bf z}^{(\bf j)}$ for all $\bf j\geq \bf 0$, which we call the \textit{Hasse derivatives}. They are defined as
\begin{equation}
    \partial_{\bf z}^{(\bf j)}\bf z^{\bf k} \eqdef \binom{\bf k}{\bf j}\bf z^{\bf k-\bf j},\qquad \partial_{\bf z}^{(\bf j)}\Omega_{\bf z}^{\bf k} \eqdef (-1)^{|\bf j|}\binom{\bf k+\bf j}{\bf j}\Omega_{\bf z}^{\bf k+\bf j}.
\end{equation}
As a basic convention, we denote the derivative of a product of series as $\partial^{(\bf k)}_{\bf z}[A(\bf z)B(\bf z)]$; this is not to be confused with the product $\partial^{(\bf k)}_{\bf z}A(\bf z)B(\bf z) = (\partial^{(\bf k)}_{\bf z}A(\bf z))B(\bf z)$.
\begin{lem}
    For all $\bf j,\bf k\geq \bf 0$ we have $\partial^{(\bf j)}_{\bf z}\partial^{(\bf k)}_{\bf z} = \binom{\bf j+\bf k}{\bf j}\partial^{(\bf j+\bf k)}_{\bf z}$. When $R$ is a $\Q$-algebra, then 
    \[\partial_{\bf z}^{(\bf k)} = \frac{1}{\bf k!}\partial^{\bf k}_{\bf z}\eqdef \prod_{i=1}^n\frac{1}{k_i!}\partial_{z_i}^{k_i}.\]
\end{lem}
\begin{notation}
    When referring to the operators $\partial^{(\bf j)}_z$ we will sometimes refer specifically to $\partial_{\bf z}^{(\bf e_i)}$ as $\partial_{z_i}$ for ease of reading.
\end{notation}
Given a series $A(\bf z)\in \cal K^{\bf z}_{\mathrm{dist}}(V)$, we split it into regular and singular parts:
\begin{equation}
    A(\bf z) = A_+(\bf z) + A_-(\bf z),\qquad A_+(\bf z) \eqdef \sum_{\bf k\geq 0}\bf z^{\bf k}A_{-\bf 1-\bf k},\qquad A_-(\bf z) \eqdef \sum_{\bf k\geq \bf 0}\Omega_{\bf z}^{\bf k}A_{\bf k}.
\end{equation}
Given $A(\bf z),B(\bf z)\in \cal K^{\bf z}_{\mathrm{dist}}(U)$ whose coefficients are all homogeneous elements of $U$, the product $A_+(\bf z)B_+(\bf z)$ is the usual product of power series, and $A_-(\bf z)B_-(\bf z) = 0$. However, the products
\begin{align*}
    A_+(\bf z)B_-(\bf z) &= \sum_{\bf k\geq \bf 0}\Omega^{\bf k}_{\bf z}\sum_{\bf j\geq \bf 0}A_{-\bf 1-\bf j}B_{\bf k+\bf j}\\
    A_-(\bf z)B_+(\bf z) &= \sum_{\bf k\geq \bf 0}\Omega^{\bf k}_{\bf z}\sum_{\bf j\geq \bf 0}(-1)^{p(A_{-\bf 1-\bf j})N}A_{\bf k+\bf j}B_{-\bf 1-\bf j}
\end{align*}
are well-defined provided that the sums of elements of $U$ on the right hand side are well-defined for all $\bf k\geq \bf 0$. This motivates the definition of fields and the normal ordered product in the next section.

For any $a(\bf z),b(\bf z)\in \cal K^{\bf z}(U)$, we have the usual Leibniz product rule
\begin{equation}
\partial^{(\bf k)}_{\bf z}[a(\bf z)b(\bf z)] = \sum_{\bf 0\leq \bf j\leq \bf k}\partial^{(\bf j)}_{\bf z}a(\bf z)\partial_{\bf z}^{(\bf k-\bf j)}b(\bf z).
\end{equation}

We also manipulate series in two sets of variables. We define the graded $R$-modules
\begin{align}
    \cal K_{\mathrm{dist}}^{\bf z,\bf w}(V) &\eqdef V[\![\bf z,\bf w]\!] \oplus (V[\![\bf z,\Omega_{\bf w}^{\bf k}]\!]_{\bf k\geq \bf 0} \oplus V[\![\bf w,\Omega_{\bf z}^{\bf k}]\!]_{\bf k\geq \bf 0}) \oplus V[\![\Omega_{\bf z}^{\bf j},\Omega_{\bf w}^{\bf k}]\!]_{\bf j,\bf k\geq \bf 0},\\
    \cal K^{\bf z,\bf w}(V) &\eqdef V[\![\bf z,\bf w]\!] \oplus (V[\![\bf z]\!][\Omega_{\bf w}^{\bf k}]_{\bf k\geq \bf 0} \oplus V[\![\bf w]\!][\Omega_{\bf z}^{\bf k}]_{\bf k\geq \bf 0}) \oplus V[\Omega_{\bf z}^{\bf j},\Omega_{\bf w}^{\bf k}]_{\bf j,\bf k\geq \bf 0},\\
    \cal K^{\bf z,\bf w}_{\mathrm{poly}}(V) &\eqdef V[\bf z,\bf w] \oplus (V[\bf z,\Omega_{\bf w}^{\bf k}]_{\bf k\geq \bf 0} \oplus V[\bf w,\Omega_{\bf z}^{\bf k}]_{\bf k\geq \bf 0}) \oplus V[\Omega_{\bf z}^{\bf j},\Omega_{\bf w}^{\bf k}]_{\bf j,\bf k\geq \bf 0}.
\end{align}
Here the first summand is in degree 0, the second in degree $N$, and the third in degree $2N$.
We impose the following graded commutation relations, which arise from the Koszul rule of signs.
\begin{equation}
    \Omega_{\bf z}^{\bf j}\Omega_{\bf w}^{\bf k} = (-1)^{N^2}\Omega_{\bf w}^{\bf k}\Omega_{\bf z}^{\bf j} = (-1)^{N}\Omega_{\bf w}^{\bf k}\Omega_{\bf z}^{\bf j},\quad \bf w^{\bf j}\Omega^{\bf k}_{\bf z} = \Omega^{\bf k}_{\bf z}\bf w^{\bf j},\quad \bf z^{\bf j}\Omega^{\bf k}_{\bf w} = \Omega^{\bf k}_{\bf w}\bf z^{\bf j}.
\end{equation}
For an $R$-algebra $U$, the $R$-module $\cal K_{\mathrm{dist}}^{\bf z,\bf w}(U)$ comes equipped with a partially defined multiplication structure (that is, the multiplication is well-defined when no infinite sums of elements of $U$ occur in the product).

We can naturally extend the residue map $\Res_{\bf z}$ for two-variable series as follows. We define for any $f\in \cal K^w_{\mathrm{dist}}(V)$
\begin{equation*}
    \Res_{\bf z}\left(\bf z^{\bf j}f(\bf w)\right) = 0,\qquad
    \Res_{\bf z}\left(\Omega^{\bf j}_{\bf z}f(\bf w)\right) = \delta_{\bf j,\bf 0}f(\bf w).
\end{equation*}

Recall that for VAs, we work with the $\delta$-distribution
\begin{equation}
    \delta(z) \eqdef \sum_{k\in \Z}z^k,
\end{equation}
and there is also the auxiliary series $\delta(z,w)\eqdef z^{-1}\delta(z/w)$, which is sometimes written suggestively as $\delta(z-w)$ due to its functional identities. We avoid using this latter notation since it conflicts with the notation for $\delta(z)$.
Here we define an analog of this auxiliary series, inspired by \cite{ravioli}:
\begin{defn}
    We define the \textit{$\Delta$-distribution} $\Delta(\bf z,\bf w)\in \cal K^{\bf z,\bf w}_{\mathrm{dist}}(\Z)$ as
    \begin{align}
        \Delta(\bf z,\bf w) &\eqdef \sum_{\bf j\geq \bf 0}\bf w^{\bf j}\Omega_{\bf z}^{\bf j} + (-1)^{N} \bf z^{\bf j}\Omega_{\bf w}^{\bf j}.
    \end{align}
\end{defn}
If one wishes they may also write $\Delta(\bf z-\bf w)$ instead of $\Delta(\bf z,\bf w)$, which is what Garner \& Williams write for $N=n=1$.
We also split this series into two parts, defining
\begin{equation}
    \Delta_-(\bf z,\bf w) \eqdef \sum_{j\geq 0}\bf w^{\bf j}\Omega_{\bf z}^{\bf j},\qquad \Delta_+(\bf z,\bf w) \eqdef (-1)^N\sum_{\bf j\geq \bf 0}\bf z^{\bf j}\Omega_{\bf w}^{\bf j}.
\end{equation}
Thus we have $\Delta(\bf z,\bf w) = \Delta_-(\bf z,\bf w) + \Delta_+(\bf z,\bf w)$.
The series $\Delta(\bf z,\bf w)$ also naturally satisfies many similar properties to that of $\delta(\bf z,\bf w)$:
\begin{lem}\label{lem:deltaidentities}
    The $\Delta$-distribution has the following properties:
    \begin{enumerate}
        \item For all $f(\bf z)\in \cal K^{\bf z,\bf w}_{\mathrm{dist}}(V)$ we have
        \[\Delta(\bf z,\bf w)f(\bf z) = \Delta(\bf z,\bf w)f(\bf w).\]
        \item For any $f(\bf z)\in \cal K^{\bf z,\bf w}_{\mathrm{dist}}(V)$ we have
        \[\Res_z\left(\Delta(\bf z,\bf w)f(\bf z)\right) = f(\bf w).\]
        \item $\Delta_\pm(\bf z,\bf w) = (-1)^{N}\Delta_\mp(\bf w,\bf z)$.
        \item $(\partial_{\bf z}^{(\bf k)}+(-1)^{|\bf k|}\partial_{\bf w}^{(\bf k)})\Delta_\pm(\bf z,\bf w) = 0$ for all $\bf k\geq \bf 0$.
        \item $(\bf z-\bf w)^{\bf j}\partial^{(\bf k)}_w\Delta_\pm(\bf z,\bf w) = \partial^{(\bf k-\bf j)}_w\Delta_\pm(\bf z,\bf w)$ if $\bf k\geq \bf j\geq \bf 0$, and 0 otherwise.
        \item For all $A\in \cal K_{\mathrm{dist}}(V)$ and $\bf k\geq \bf 0$ we have
        \begin{equation}
            \partial^{(\bf k)}_{\bf w}\Delta(\bf z,\bf w)A(\bf z) = \partial^{(\bf k)}_{\bf w}\Delta(\bf z,\bf w)\sum_{\bf 0\leq \bf j\leq \bf k}(\bf z-\bf w)^{\bf j}\partial_{\bf w}^{(\bf j)}A(\bf w).
        \end{equation}
        \item For all $\bf p,\boldsymbol\ell\geq \bf 0$ we have
        \[\partial_{\bf y}^{(\boldsymbol\ell)}\Delta(\bf x,\bf y) \partial_{\bf z}^{(\bf p)}\Delta(\bf y,\bf z) = \sum_{\bf 0\leq \bf k\leq \bf p}\binom{\bf k+\boldsymbol\ell}{\boldsymbol\ell}\partial_{\bf z}^{(\boldsymbol\ell+\bf k)}\Delta(\bf x,\bf z)\partial_{\bf z}^{(\bf p-\bf k)}\Delta(\bf y,\bf z).\]
    \end{enumerate}
\end{lem}
\begin{proof}
    Expand both sides of (1), letting $f(\bf z) = \sum_{\bf k\geq \bf 0}\bf z^{\bf k}f_{-\bf 1-\bf k} + \Omega_{\bf z}^{\bf k}f_{\bf k}$:
    \begin{align*}
        \Delta(\bf z,\bf w)f(\bf z) &= \sum_{\bf j,\bf k\geq \bf 0}\left(\bf w^{\bf k}\Omega_{\bf z}^{\bf k} + (-1)^{N} \bf z^{\bf k}\Omega_{\bf w}^{\bf k}\right)f_{-\bf 1-\bf j}\bf z^{\bf j} + \left(\bf w^{\bf k}\Omega_{\bf z}^{\bf k} + (-1)^{N} \bf z^{\bf k}\Omega_{\bf w}^{\bf k}\right)f_{\bf j}\Omega_{\bf z}^{\bf j}\\
        &= \sum_{\bf j,\bf k\geq \bf 0}\left(\bf w^{\bf k}\Omega_{\bf z}^{\bf k-\bf j} + (-1)^{N} \bf z^{\bf j+\bf k}\Omega_{\bf w}^{\bf k}\right)f_{-\bf 1-\bf j} + (-1)^{N} \Omega_{\bf z}^{\bf j-\bf k}\Omega_{\bf w}^{\bf k}f_{\bf j}\\
        \Delta(\bf z,\bf w)f(\bf w) &= \sum_{\bf j,\bf k\geq \bf 0}\left(\bf w^{\bf k}\Omega_{\bf z}^{\bf k} + (-1)^{N} \bf z^{\bf k}\Omega_{\bf w}^{\bf k}\right)f_{-\bf 1-\bf j}\bf w^{\bf j} + \left(\bf w^{\bf k}\Omega_{\bf z}^{\bf k} + (-1)^{N} \bf z^{\bf k}\Omega_{\bf w}^{\bf k}\right)f_{\bf j}\Omega_{\bf w}^{\bf j}\\
        &= \sum_{\bf j,\bf k\geq \bf 0}\left(\bf w^{\bf j+\bf k}\Omega_{\bf z}^{\bf k} + (-1)^{N} \bf z^{\bf w}\Omega_{\bf w}^{\bf k-\bf j}\right)f_{-\bf 1-\bf j} + \Omega_{\bf w}^{\bf j-\bf k}\Omega_{\bf z}^{\bf k}f_{\bf j}.
    \end{align*}
    Using the commutation relations and reindexing the sums shows us (1). (2) follows from (1). (3), (4), and (5) follow from similar expansions and observations. (6) follows from using the product rule on $\partial^{(\bf j)}_{\bf w}[\Delta(\bf z,\bf w)A(\bf w)]$ and (5). 
    (7) is a corollary of (6).
    % \begin{align*}
    %     \partial_y^{(\ell)}\Delta(x,y) \partial_z^{(p)}\Delta(y,z)
    %      &= 
    %      \partial_z^{(p)}(\partial_y^{(\ell)}\Delta(x,y)\Delta(y,z))\\
    %      &= 
    %      \partial_z^{(p)}(\partial_z^{(\ell)}\Delta(x,z)\Delta(y,z))\\
    %      &= 
    %      \sum_{0\leq k\leq p}\partial_z^{(k)}\partial_z^{(\ell)}\Delta(x,z)\partial_z^{(p-k)}\Delta(y,z)\\
    %      &= 
    %      \sum_{0\leq k\leq p}\binom{k+\ell}{\ell}\partial_z^{(\ell+k)}\Delta(x,z)\partial_z^{(p-k)}\Delta(y,z).
    % \end{align*}
\end{proof}
\begin{remark}
    Consider the analog of property (5) for formal calculus for VAs. We have $(x-y)(x-y)^{-1} = 1$ and $(x-y)(-y+x)^{-1} = 1$, but (5) tells us that $(z_i-w_i)\Delta_\pm(\bf z,\bf w) = 0$ for all $i$. This implies that certain structural aspects of cohomological vertex algebras will be noticeably different from classical vertex algebras.
\end{remark}

\subsection{Taylor expansions and correlation functions}\label{subsec:taylor}
The configuration space of two formal variables $V[\![z,w]\! ][z^{-1},w^{-1},(z-w)^{-1}]$ with coefficients in $V$ plays a key role in the theory of vertex algebras. One may think of this as the global sections of a scheme representing the configuration space of two points in $D_1^\circ$ (with values in $V$).
We define $\cal K^{\bf z,\bf w,\bf z-\bf w}(V)$ to be the graded $R$-module $V[\![\bf z,\bf w]\!][\Omega_{\bf z}^{\bf j},\Omega_{\bf w}^{\bf k},\Omega_{\bf z-\bf w}^{\bf m}]_{\bf j,\bf k,\bf m\geq \bf 0}$ where the degree $N$ forms $\{\Omega_{\bf z}^{\bf j},\Omega_{\bf w}^{\bf k},\Omega_{\bf z-\bf w}^{\bf m}\}_{\bf j,\bf k,\bf m\geq \bf 0}$ are subject to the following relations:
\begin{enumerate}
    \item $\bf x^{\bf j}\Omega_{\bf x}^{\bf k} = \Omega_{\bf x}^{\bf k-\bf j}$ when $\bf k\geq \bf j$ and 0 otherwise for $\bf x=\bf z,\bf w,\bf z-\bf w$.
    \item $\Omega^{\bf j}_{\bf x}\Omega^{\bf k}_{\bf x} = 0$ for all $\bf j,\bf k\geq \bf 0$ and $\bf x=\bf z,\bf w,\bf z-\bf w$.
    \item $(\Omega_{\bf w}^{\bf 0}-\Omega_{\bf z}^{\bf 0})\Omega^{\bf 0}_{\bf z-\bf w} = \Omega^{\bf 0}_{\bf w}\Omega^{\bf 0}_{\bf z}$.
    \item $\partial^{(\bf k)}_{\bf z}\Omega_{\bf z-\bf w}^{\bf j} = (-1)^{|\bf k|}\partial^{(\bf k)}_{\bf w}\Omega_{\bf z-\bf w}^{\bf j} = (-1)^{|\bf k|}\binom{\bf j+\bf k}{\bf k}\Omega_{\bf z-\bf w}^{\bf j+\bf k}$.
\end{enumerate}
We also include all the relations produced by applying $\partial_{\bf w}^{(\bf k)}$ or $\partial_{\bf z}^{(\bf k)}$ to (3).
\begin{lem}\label{lem:multivarembed}
    The assignment
    \begin{equation}
        i_{\bf z,\bf w} \colon \Omega^{\bf m}_{\bf z-\bf w} \mapsto \sum_{\bf k\geq \bf 0}\binom{\bf m+\bf k}{\bf k}\bf w^{\bf k}\Omega_{\bf z}^{\bf m+\bf k} = \partial_{\bf w}^{(\bf m)}\Delta_-(\bf z,\bf w)
    \end{equation}
    defines a graded $R$-algebra embedding $\cal K^{\bf z,\bf w,\bf z-\bf w}\to R\langle\!\langle \bf z\rangle\!\rangle\langle\!\langle \bf w\rangle\!\rangle$. Similarly the assignment
    \begin{equation}
        i_{\bf w,\bf z} \colon \Omega^{\bf m}_{\bf z-\bf w} \mapsto -\sum_{\bf k\geq \bf 0}\bf z^{\bf k}(-1)^{|\bf m|+N}\binom{\bf m+\bf k}{\bf k}\Omega^{\bf m+\bf k}_{\bf w} = -\partial_{\bf w}^{(\bf m)}\Delta_+(\bf z,\bf w)
    \end{equation}
    defines a graded $R$-algebra embedding $\cal K^{\bf z,\bf w,\bf z-\bf w}\to R\langle\!\langle \bf w\rangle\!\rangle\langle\!\langle \bf z\rangle\!\rangle$. 
\end{lem}
One may also construct a graded $R$-algebra embedding $i_{\bf w,\bf z-\bf w}$ from $\cal K^{\bf z,\bf w,\bf z-\bf w}$ into $R\langle\!\langle \bf w\rangle\!\rangle\langle\!\langle \bf z-\bf w\rangle\!\rangle$ by writing $\bf z = \bf w + (\bf z-\bf w)$ and expanding in $\bf z-\bf w$.

Now we discuss how to make sense of the expression $\Omega_{-\bf z}^{\bf k}$ for $\bf k\geq \bf 0$. The expansions described above should be invariant under our choice of notation in the sense that $\Omega^{\bf m}_{\bf z-\bf w}$ and $\Omega^{\bf m}_{-\bf w+\bf z}$ should be the same element, meaning that $i_{\bf w,\bf z}\Omega^{\bf m}_{-\bf w+\bf z} = i_{\bf w,\bf z}\Omega^{\bf m}_{\bf z-\bf w}$:
\[i_{\bf w,\bf z}\Omega^{\bf m}_{-\bf w+\bf z} = \sum_{\bf k\geq \bf 0}\bf z^{\bf k}\partial_{-\bf w}^{(\bf k)}\Omega_{-\bf w}^{\bf m} = \sum_{\bf k\geq \bf 0}\bf z^{\bf k}(-1)^{|\bf k|}\binom{\bf m+\bf k}{\bf k}\Omega_{-\bf w}^{\bf m+\bf k}.\]
Comparing the above expressions, we obtain a natural interpretation of $\Omega_{-\bf w}^{\bf m}$:
\begin{defn}
For all $\bf m\geq \bf 0$ we define
\begin{equation}\label{eq:negativemodedef}
    \Omega^{\bf m}_{-\bf z} = (-1)^{|\bf m|+N+1}\Omega^{\bf m}_{\bf z}.
\end{equation}
\end{defn}
Under this convention, we may naturally interpret
\begin{equation}
    i_{\bf z,\bf w}\Omega^{\bf m}_{\bf z-\bf w} = e^{-\bf w\partial_{\bf z}}\Omega^{\bf m}_{\bf z},
\end{equation}
where $e^{\bf w\partial_{\bf z}}\eqdef\sum_{\bf k\geq \bf 0}\bf w^{\bf k}\partial_{\bf z}^{(\bf k)}$.
We may also discuss the $\Delta$-distribution for different arguments. For example, we have
\begin{equation}
    \Delta_-(-\bf w,-\bf z) = \sum_{\bf k\geq \bf 0}(-\bf w)^{\bf k}\Omega^{\bf k}_{-\bf z} 
    = (-1)^{N+1}\sum_{\bf k\geq \bf 0}\bf w^{\bf k}\Omega^{\bf k}_{\bf z} = -\Delta_+(\bf z,\bf w).
\end{equation}
Thus we have $\Delta(\bf x,\bf y) = \Delta_-(\bf x,\bf y) - \Delta_-(-\bf y,-\bf x)$.

The Jacobi identity for VAs features the three-variable $\delta$-distribution
\[\delta(x-y,z) = e^{-y\partial_x}[z^{-1}\delta(x/z)].\]
This three-variable series has two important identities:
\begin{align*}
    \delta(x-z,y) &= \delta(y+z,x),\\
    \delta(x-z,y) &= \delta(x-y,z) - \delta(-y+x,z).
\end{align*}
Similarly, we define a three-variable $\Delta$-distribution as
\begin{equation}
    \Delta(\bf x-\bf y,\bf z) \eqdef e^{-\bf y\partial_{\bf x}}\Delta(\bf x,\bf z).
\end{equation}
where again $e^{-\bf y\partial_{\bf x}} = \sum_{\bf k\geq \bf 0}(-\bf y)^{\bf k}\partial_{\bf x}^{(\bf k)}$. This analog satisfies nearly the same identities:
\begin{lem}
    We have the following identities:
    \begin{align}
        \Delta(\bf x-\bf z,\bf y) &= (-1)^N\Delta(\bf y+\bf z,\bf x),\\
        \Delta(\bf x-\bf z,\bf y) &= \Delta(\bf x-\bf y,\bf z) - \Delta(-\bf y+\bf x,\bf z).
    \end{align}
\end{lem}

\subsection{Fields and locality}
Before giving the definition of a CVA, we define the space of fields for our state/field correspondence $Y$.
\begin{defn}\label{defn:field}
A \textit{cohomological field} on $V$ (or just a field) is an element
\[A(\bf z) = \sum_{\bf j\geq \bf 0}\bf z^{\bf j}A_{(-\bf 1-\bf j)} + \Omega_{\bf z}^{\bf j}A_{(\bf j)} \in \cal K^{\bf z}_{\mathrm{dist}}(\End(V))\]
such that for all $v\in V$ there exists $\bf K = \bf K_{A,v}\geq 0$ such that $A_{(\bf k)}v = 0$ unless $\bf k\leq \bf K$. Equivalently we require that $A(\bf z)v\in V\langle\!\langle \bf z \rangle\!\rangle$ for all $v\in V$. We denote the space of $\End(V)$-valued cohomological fields as $\cal E_c(V)$.

We say that a field $A(\bf z)\in \cal E_c(V)$ is \textit{homogeneous of cohomological degree $p(A)$} if for all $\bf j\geq \bf 0$ the endomorphism $A_{(-\bf 1-\bf j)}$ is of degree $p(A)$ and $A_{(\bf j)}$ is of degree $p(A)-N$.
\end{defn}
We may define a notion of locality for these fields. Recall that in a $\Z$-graded algebra $U$ we define the graded commutator $[x,y] = xy-p(x,y)yx$ for homogeneous elements $x,y\in U$. We extend this definition $R$-bilinearly to every pair of elements of $U$.
\begin{defn}\label{defn:locality}
    Given $f(\bf z,\bf w)\in \cal K^{\bf z,\bf w}_{\mathrm{dist}}(V)$, we say it is \textit{local with respect to $(\bf z,\bf w)$} if it can be written as a sum
    \[f(\bf z,\bf w) = \sum_{0\leq j\leq K}\partial^{(\bf j)}_{\bf w}\Delta(\bf z,\bf w)g^{j}(\bf w)\]
    for some $K\in \N^n$ and $g^j(\bf w) = \Res_z\left((\bf z-\bf w)^jf(\bf z,\bf w)\right)$.

    Given two fields $A(\bf z),B(\bf w)$, we say they are \textit{mutually local} if the graded commutator $[A(\bf z),B(\bf w)]v]\in \cal K^{\bf z,\bf w}_{\mathrm{dist}}(V)$ is local for all $v\in V$. Equivalently, the series $[A(\bf z),B(\bf w)]\in \cal K^{\bf z,\bf w}_{\mathrm{dist}}(\End(V))$ is local. Here we define
    \[A(\bf w)_{(\bf j)}B(\bf w) = \Res_z\left((\bf z-\bf w)^j[A(\bf z),B(\bf w)]\right).\]
\end{defn}
Note that $A(\bf w)_{(\bf j)}B(\bf w)$ is a homogeneous field of degree $p(A)+p(B)$.

A basic operation one may try to perform with two fields $A(\bf z)$ and $B(\bf z)$ in the same set of variables $\bf z$ is take their product. It turns out that this is not well-defined in general, and it is preferable to work with the following product.
\begin{defn}
    Given homogeneous fields $A,B\in \cal E_c(V)$, we define their \textit{normal ordered product} to be
    \begin{equation}
	\normord{A(\bf z)B(\bf w)} = A(\bf z)_+B(\bf w) + p(A,B)B(\bf w)A(\bf z)_-.
    \end{equation}
\end{defn}
It is straightforward to check that the normal ordered product is well-defined for $\bf z=\bf w$. That is, the series $\normord{AB}\!(\bf z) = \normord{A(\bf z)B(\bf z)}$ is a homogeneous field of degree $p(A)+p(B)$.
We sometimes write $A(\bf z)_{(-\bf 1-\bf k)}B(\bf z) \eqdef\, \normord{\partial_{\bf z}^{(\bf k)}A(\bf z)B(\bf z)}$.
\begin{prop}
    Let $A(\bf z),B(\bf x),C(\bf w)$ be pairwise mutually local homogeneous fields. Then we have
    \begin{align}
        \normord{A(\bf z)B(\bf z)}\! &= p(A,B)\!\normord{B(\bf z)A(\bf z)},\\
        \normord{A(\bf z)\!\normord{B(\bf z)C(\bf z)}}\! &= \ \normord{\normord{A(\bf z)B(\bf z)}\!C(\bf z)}.
    \end{align}
\end{prop}
The proof is the same as in \cite[Proposition 2.2.6]{ravioli}, which amounts to observing that $A(\bf z)_-B(\bf z)_-=0$ and $[A(\bf z)_+,B(\bf z)_+] = 0$ (likewise for the pairs $(A,C)$ and $(B,C)$). This is a strong departure from the behavior of the normal ordered product of fields for VAs, where neither of the above properties usually hold.

Here are two equivalent presentations of mutual locality.
\begin{lem}
    Two homogeneous fields $A,B\in \cal E_c(V)$ are mutually local if and only if for the same choice of $\bf K\in \N^n$ we have the following operator product expansions (OPEs)
    \begin{align*}
        A(\bf z)B(\bf w) &=\, \normord{A(\bf z)B(\bf w)} + \sum_{\bf 0\leq \bf j\leq \bf K}\partial^{(\bf j)}_{\bf w}\Delta_-(\bf z,\bf w)A(\bf w)_{(\bf j)}B(\bf w)\\
        p(A,B)B(\bf w)A(\bf z) &=\, \normord{A(\bf z)B(\bf w)} - \sum_{\bf 0\leq \bf j\leq \bf K}\partial^{(\bf j)}_{\bf w}\Delta_+(\bf z,\bf w)A(\bf w)_{(\bf j)}B(\bf w).
    \end{align*}
\end{lem}
\begin{proof}
    These formulas follow from the following identities:
    \begin{align*}
        A(\bf z)B(\bf w) &=\, \normord{A(\bf z)B(\bf w)} + [A(\bf z)_-,B(\bf w)]\\
        p(A,B)B(\bf w)A(\bf z) &=\, \normord{A(\bf z)B(\bf w)} - [A(\bf z)_+,B(\bf w)].
    \end{align*}
\end{proof}
Here is a more analytic interpretation of the above notion of locality:
\begin{prop}\label{prop:localanalytic}
    Let $A,B\in \cal E_c(V)$ be homogeneous. Then $A(\bf z),B(\bf w)$ are mutually local if and only if for all $v\in V$ the series
    \[A(\bf z)B(\bf w)v\in V\langle\!\langle \bf z\rangle\!\rangle\langle\!\langle \bf w\rangle\!\rangle \quad \mathrm{and}\quad p(A,B)B(\bf w)A(\bf z)v\in V\langle\!\langle \bf w\rangle\!\rangle\langle\!\langle \bf z\rangle\!\rangle\]
    are expansions of the same element $f_v\in \cal K^{\bf z,\bf w,\bf z-\bf w}(V)$ such that there exists $\bf K\geq \bf 0$ (independent of $v$) such that the coefficient in front of $\Omega_{\bf z-\bf w}^{\bf k}$ in $f_v$ is zero unless $\bf k\leq \bf K$.
\end{prop}

\begin{proof}
Suppose $A,B$ are mutually local. Consider the OPEs from the previous lemma.
We just need to examine the second term in the OPE for $p(A,B)B(\bf w)A(\bf z)v$ since the normal ordered product term is in $\cal K^{\bf z,\bf w}(V)$.
Since $\partial^{(\bf k)}_{\bf w}\Delta_-(\bf z,\bf w)$ and $-\partial^{(k)}_w\Delta_+(\bf z,\bf w)$ are the expansions of $\Omega_{\bf z-\bf w}^{\bf k}$ in their respective domains, we obtain that $A(z)B(w)v$ and $p(A,B)B(\bf w)A(\bf z)v$ are the expansions of the same element in $\cal K^{\bf z,\bf w,\bf z-\bf w}(V)$. The pole is uniformly bounded by locality since 
$\sum_{\bf k\geq \bf 0}\partial^{(\bf k)}_{\bf w}\Delta(\bf z,\bf w)C^{\bf k}(\bf w)$ is a finite sum. 

Conversely, we may uniquely write $f_v\in \cal K^{\bf z,\bf w,\bf z-\bf w}(V)$ as a sum
\[f_v = g_v(\bf z,\bf w) + \sum_{\bf 0\leq \bf k\leq \bf K}\Omega^{\bf k}_{\bf z-\bf w}f^{\bf k}_v(\bf w),\]
where $f^{\bf k}_v(\bf w)\in V\langle\!\langle \bf w\rangle\!\rangle$ and $g_v\in \cal K^{\bf z,\bf w}(V)$. By applying the embeddings from the previous lemma, we obtain the embeddings
\begin{align*}
    f_v &\mapsto g_v(\bf z,\bf w) + \sum_{\bf 0\leq \bf k\leq \bf K}\partial_{\bf w}^{(\bf k)}\Delta_-(\bf z,\bf w)f^{\bf k}_v(\bf w)\\
    f_v &\mapsto g_v(\bf z,\bf w) - \sum_{\bf 0\leq \bf k\leq \bf K}\partial_{\bf w}^{(\bf k)}\Delta_+(\bf z,\bf w)f^{\bf k}_v(\bf w).
\end{align*}
Taking the difference in $\cal K^{\bf z,\bf w}_{\mathrm{dist}}(V)$, we obtain
\[[A(\bf z),B(\bf w)]v = \sum_{\bf 0\leq \bf k\leq \bf K}\partial^{(\bf k)}_{\bf w}\Delta(\bf z,\bf w)f^{\bf k}_v(\bf w),\]
which holds for all $v$. Since $f^{\bf k}_v(\bf w)$ is $R$-linear in $v$, we conclude that $f^{\bf k}_v(\bf w)$ must be of the form $C^{\bf k}(\bf w)v$ for some $C^{\bf k}(\bf w)\in \cal K^{\bf w}_{\mathrm{dist}}(\End(V))$.
\end{proof}
For shorthand, the condition on the existence of $\bf K\geq \bf 0$ may be rephrased as saying that the order of the pole at $\bf z-\bf w$ is uniformly bounded for all $v\in V$.

Now we provide a version of Dong's lemma, which is slightly weaker than the usual statement for VAs (not that it will matter).
\begin{lem}[Dong-Li lemma]\label{lem:Dongpoly}
    Let $A,B,C$ be fields on $V$ with some $\bf r\geq \bf 0$ such that for all $i$ we have
    \[(z_i-x_i)^{r_i}[A(\bf z),B(\bf x)] = (z_i-w_i)^{r_i}[A(\bf z),C(\bf w)] = (x_i-w_i)^{r_i}[B(\bf x),C(\bf w)] = 0.\]
Then the following hold:
\begin{enumerate}
    \item For all $i$ and $k\geq 0$ we have $(z_i-w_i)^{3r_i}[A(\bf z),B(\bf w)_{(\bf k)}C(\bf w)]=0$.
    \item 
    Further suppose that $A,B$ and $A,C$ are mutually local, with
    \begin{align*}
        [A(\bf z),B(\bf x)] &= \sum_{\bf 0\leq \bf m\leq \bf r-\bf 1}\partial^{(\bf m)}_{\bf x}\Delta(\bf z,\bf x)D^{\bf m}(\bf x),\\
        [A(\bf z),C(\bf w)] &= \sum_{\bf 0\leq \bf m\leq \bf r-\bf 1}\partial^{(\bf m)}_{\bf w}\Delta(\bf z,\bf w)E^{\bf m}(\bf w).
        % [B(x),C(w)] &= \sum_{0\leq m\leq r-1}\partial^{(m)}_w\Delta(x,w)F^m(w).
    \end{align*}
    Then the pair $A(\bf z)$ and $\normord{BC}\!(\bf w)$ is mutually local, with commutator
\begin{equation}\label{eq:normordbracket}
\begin{aligned}
    &[A(\bf z),\normord{BC}\!(\bf w)] \\&= \sum_{m\geq 0}\partial^{(\bf m)}_w\Delta(\bf z,\bf w)\Big(\normord{D^{\bf m}C}(\bf w) + (-1)^{p(B)(p(A)+N)}\normord{BE^{\bf m}}(\bf w)\Big).
\end{aligned}
\end{equation}
\end{enumerate}
\end{lem}
\begin{proof}
For statement (1), we may carry out a similar argument to the original proof for VAs \cite[Proposition 3.2.7]{LieSuperalgebras} by expanding 
\[(z_i-w_i)^{3r_i} = \sum_{s_i=0}^{2r_i}\binom{2r_i}{s_i}(z_i-x_i)^{s_i}(x_i-w_i)^{2r_i-s_i}(z_i-w_i)^{r_i}.\]
The original proof used a factor of 4 instead of 3; this is only necessary if one wants to consider a more general version of this statement for normal ordered products.
For statement (2) we expand to get
\begin{align*}
    &[A(\bf z),\normord{BC}\!(\bf w)]\\
    % &= [A(z) , B(w)_+C(w)]  + p(B,C)[A(z),C(w)_+B(w)_-]\\
    &= [A(\bf z),B(\bf w)_+]C(\bf w) + p(A,B)B(\bf w)_+[A(\bf z),C(\bf w)]\\
    &\qquad +p(B,C)[A(\bf z),C(\bf w)_+]B(\bf w)_- + p(B,C)p(A,C)C(\bf w)_+[A(\bf z),B(\bf w)_-].
\end{align*}
We expand two of the terms on the right hand side together:
\begin{align*}
    &p(A,B)B(\bf w)_+[A(\bf z),C(\bf w)] + p(B,C)[A(\bf z),C(\bf w)_+]B(\bf w)_-\\
    % &= p(A,B)\sum_{m\geq 0}\partial^{(m)}_w\Delta(\bf z,\bf w)\left((-1)^{p(B)N}B(w)_+E^m(w)+p(A,B)p(B,C)E^m(w)B(w)_-\right)\\
    &=(-1)^{p(B)(p(A)+N)}\sum_{m\geq 0}\partial^{(\bf m)}_{\bf w}\Delta(\bf z,\bf w) \normord{BE^{\bf m}}\!(\bf w).
\end{align*}
As for the second term, we have
\begin{align*}
    &[A(\bf z),B(\bf w)_+]C(\bf w) + p(B,C)p(A,C)C(\bf w)_+[A(\bf z),B(\bf w)_-]\\
    % &=\sum_{m\geq 0}\partial_w^{(m)}\Delta_-(\bf z,\bf w)D^m(w)_+C(w) \\
    % &\qquad + (-1)^{(p(A)+p(B))p(C)}C(w)_+\left(\partial_w^{(m)}\Delta_-(\bf z,\bf w)D^m(w)_- + \partial_w^{(m)}\Delta_+(\bf z,\bf w)D^m(w)_+\right)\\
    % &=\sum_{m\geq 0}\partial_w^{(m)}\Delta_-(\bf z,\bf w)D^m(w)_+C(w) \\
    % &\qquad + (-1)^{(p(A)+p(B)+N)p(C)}\partial_w^{(m)}\Delta_-(\bf z,\bf w)C(w)_+D^m(w)_-\\
    % &\qquad + (-1)^{(p(A)+p(B)+N)p(C)}\partial_w^{(m)}\Delta_+(\bf z,\bf w)C(w)_+D^m(w)_+\\
    &=\sum_{\bf m\geq \bf 0}\partial_{\bf w}^{(\bf m)}\Delta(\bf z,\bf w)\normord{D^{\bf m}C}(\bf w) - \partial^{(\bf m)}_{\bf w}\Delta_+(\bf z,\bf w)[D^{\bf m}(\bf w)_+,C(\bf w)_+].
\end{align*}
By the first statement of the lemma, we have that $[C(\bf z)_+,D^{\bf m}(\bf w)_+]=0$, and since these are power series we can set $\bf z=\bf w$. Thus the second term is 0, proving (2).
\end{proof}
Equation \ref{eq:normordbracket} will be our primary tool to compute brackets for examples of CVAs.

\section{Cohomological vertex algebras}\label{Sec3}
Here, we give the main definition of a CVA and state its various properties. The main result we will build up to is the Jacobi identity (Theorem \ref{thm:Jacobi}). Once again, we fix a pair of integers $N,n$ such that $n\in \Z_{\geq 1}$, and we recall that we use the shorthand $\bf z = (z_1,\ldots,z_n)$. We also write $\bf k\geq \bf j$ to mean $k_i\geq j_i$ for all $i=1,\ldots,n$.

\subsection{The main definitions}\label{subsec:CVA}
The key object in the definition of a CVA is the \textit{state-field correspondence} $Y$, which maps every element $a$ of a $\Z$-graded $R$-module $V$ to a field $Y(a,z)$. When we provide examples of cohomological vertex algebras in Section \ref{Sec5}, we will sometimes write $a(z)$ to mean $Y(a,z)$.
\begin{defn}\label{defn:CVA}
    Let $R$ be a commutative ring, and fix $N\in \Z$ and $n\in \Z_{\geq 1}$. An \textit{$N$-shifted cohomological vertex algebra in $n$ variables} (or just a \textit{$(n,N)$-CVA}) is the data $(V,Y,e^{\bf z\cdot \bf T},\mathds 1)$ consisting of the following:
    \begin{enumerate}
        \item (State space) A $\Z$-graded $R$-module $V = \bigoplus_{r\in \Z}V^r$.
        \item (Vertex operator map) A degree 0 $R$-linear map $Y \colon V\to \cal K^{\bf z}_{\mathrm{dist}}(\End(V))$:
        \[a\mapsto Y(a,\bf z) = \sum_{\bf j\geq \bf 0}\bf z^{\bf j}a_{(-\bf 1-\bf j)} + \Omega^{\bf j}_{\bf z}a_{(\bf j)}.\]
        That is, if $a\in V$ is homogeneous, then $Y(a,\bf z)$ is a homogeneous cohomological field of degree $p(a)$ in the sense of Definition \ref{defn:field}. That is, $a_{(\bf j)}$ is a degree $p(a)-N$ operator and $a_{(-\bf 1-\bf j)}$ is a degree $p(a)$ operator.
        \item (Translation operator) A collection $(T^{(\bf j)})_{\bf j\geq \bf 0}\subset \End(V)$ of degree 0 endomorphisms such that $T^{(\bf 0)} = \id_V$, which we write together in a power series as $e^{\bf z\cdot \bf T} = \sum_{\bf j\geq \bf 0}\bf z^{\bf j}T^{(\bf j)}\in \End(V)[\![\bf z]\!]$.
        \item (Vacuum) A distinguished degree 0 element $\bf 1\in V^0$.
    \end{enumerate}
    This data satisfies the following conditions for all $a,b\in V$:
    \begin{enumerate}[label = (V\arabic*)]
        \item For all $a\in V$ we have $Y(a,\bf z)\in \cal E_c(V)$ in the sense of Definition \ref{defn:field}. That is, for all $a,b\in V$, we have $Y(a,\bf z)b\in V\langle\!\langle \bf z\rangle\!\rangle$.
        \item $Y(\bf 1,\bf z) = \id_V$.
        \item $Y(a,\bf z)\mathds 1 = e^{\bf z\cdot \bf T}a$.
        \item We have
        \begin{equation}\label{eq:transcov}
            e^{\bf w\cdot \bf T}Y(a,\bf z)e^{-\bf w\cdot \bf T} = Y(a,\bf z+\bf w) = \sum_{\bf k\geq \bf 0}\bf w^{\bf k}\partial_{\bf z}^{(\bf k)}Y(a,\bf z).
        \end{equation}
        Here we will always write $Y(a,\bf z+\bf w)$ to mean $i_{\bf z,\bf w}Y(a,\bf z+\bf w)$ unless stated otherwise.
        \item For all homogeneous elements $a,b\in V$, the fields $Y(a,\bf z)$ and $Y(b,\bf w)$ are mutually local in the sense of Definition \ref{defn:locality}.
    \end{enumerate}
\end{defn}
Axioms V2, V3, and V4 mirror the usual vacuum/translation axioms for VAs.
We direct the reader to Proposition \ref{prop:Qalg} in case they wish to see a more familiar set of axioms when $R$ is a field of characteristic 0. 
We sometimes refer to axiom V4 as \textit{translation covariance} and axiom V5 as \textit{weak commutativity}.
We will leave the integers $(n,N)$ implicit and simply say `CVA' when clear in context.

One may give natural definitions of ideals, morphisms, subalgebras, tensor products, and direct sums similar to how it is given for VAs. We omit the definitions as they are identical to the constructions for vertex superalgebras.
% \begin{defn}
%     Let $W,I\leq V$ be $R$-submodules of $V$ invariant under $T^{(j)}$ for all $j\geq 0$. $W$ is said to be a \textit{cohomological vertex subalgebra} (or just a subalgebra) of $V$ if $Y(a,z)b\in W\langle\!\langle z\rangle\!\rangle$ whenever $a,b\in W$. $I$ is said to be a \textit{cohomological vertex algebra ideal} (or just an ideal) if $Y(a,z)b\in I\langle\!\langle z\rangle\!\rangle$ whenever $a\in I$ and $b\in V$.
% \end{defn}
% \begin{defn}
%     A \textit{morphism} of cohomological vertex algebras
%     \[\rho \colon (V_1,Y_1,e^{\bf z\cdot \bf T_1},\mathds 1_1) \to (V_2,Y_2,e^{\bf z\cdot \bf T_2},\mathds 1_2)\]
%     is an $R$-linear degree 0 map $\rho \colon V_1\to V_2$ such that for all $a,b\in V$ we have
%     \[\rho(\mathds 1_1) = \mathds 1_2,\qquad e^{\bf z\cdot \bf T_2}\rho(a) = \rho(e^{\bf z\cdot \bf T_1}a),\qquad \rho(Y_1(a,z)b) = Y_2(\rho(a),z)\rho(b).\]
%     We will often refer to such morphisms as just $\rho \colon V_1\to V_2$.
% \end{defn}
% The kernel of a morphism $\rho \colon V_1\to V_2$ is an ideal of $V_1$, and the image is a subalgebra of $V_2$. If $J\subseteq V$ is an ideal then the quotient $V/J$ has the natural structure of a cohomological vertex algebra since $Y(a,z)b\in J\langle\!\langle z\rangle\!\rangle$ implies $Y(b,z)a\in J\langle\!\langle z\rangle\!\rangle$ by the skew-symmetry identity (Proposition \ref{prop:skewsymmetry}).

% Given CVAs $V_1,\ldots,V_d$, one may define a natural CVA structure on $V_1\otimes \ldots\otimes V_d$ similar to that of the tensor product of VAs.
% The direct sum of CVAs can be defined similarly.

The first additional structure we consider on a CVA is a compatible grading on the underlying $R$-module, which we call a spin grading (following \cite{ravioli}). This is analogous to the usual grading that is considered for VAs.
\begin{defn}\label{defn:spin}
    Let $L\in \Z_{\geq 1}$. A \textit{spin grading} on a cohomological vertex algebra $V$ is an additional (non-cohomological) grading $V = \bigoplus_{\bf s\in (\tfrac{1}{L}\Z)^n}V^{(\bf s)}$ such that:
    \begin{enumerate}
        \item $\mathds 1\in V^{(\bf 0)}$.
        \item $Y$ is a spin degree 0 map.
        \item $T^{(\bf j)}$ is a spin degree $j$ map, that is $T^{(\bf j)}V^{(\bf k)}\subset V^{(\bf k+\bf j)}$ for all $\bf k$.
        \item For all $\bf j\geq \bf 0$ and $a\in V^{(\bf s)}$ the mode $a_{(\bf j)}$ is a spin degree $\bf s-\bf j-\bf 1$ map, and $a_{(-\bf 1-\bf j)}$ is a spin degree $\bf s+\bf j$ map.
    \end{enumerate}
    We denote the spin degree of an element or mode via $\deg$, and we denote the $i$th component of the spin degree via $\deg_i$.
    Furthermore, we say that an \textit{$\N^n$-grading} of $V$ is a spin grading for $L=1$ such that $V^{(\bf s)} = 0$ unless $\bf s\geq \bf 0$.
\end{defn}
Most of the gradings we will consider in this paper are $\N^n$-gradings. 
% Given an $\N^n$-grading $\deg$, we may also descend to an $\N$-grading $|\!\deg\!|$ given by
% \begin{equation}
% |\!\deg\!|(a) = |\deg (a)| = \sum_i\deg_i(a).
% \end{equation}
% Axiom V1 may then be replaced with $a_{(k)}b=0$ whenever $|k| > n(\max_i\{K_i\})$ since that implies $k_j > K_j$ for some $j$ (here $K=K_{a,b}$ from axiom V1).

\subsection{Basic properties of cohomological vertex algebras}
Let $V = (V,Y,e^{\bf z\cdot \bf T},\mathds 1)$ be a cohomological vertex algebra over a ring $R$. We prove some general properties of these objects. 
The first property we note is that the generator of translations $e^{\bf z\cdot \bf T}$ behaves as its notation suggests:
\begin{lem}\label{lem:exp}
    We have $e^{\bf z\cdot \bf T}e^{\bf w\cdot \bf T} = e^{\bf w\cdot \bf T}e^{\bf z\cdot \bf T} = e^{(\bf z+\bf w)\cdot \bf T}$.
\end{lem}
\begin{proof}
    By V3 we have
    \[Y(a,\bf z+\bf w)\mathds 1 = e^{(\bf z+\bf w)\cdot \bf T}a.\]
    On the other hand, by V4 and V2 we have
    \[Y(a,\bf z+\bf w)\mathds 1 = e^{\bf w\cdot \bf T}Y(a,\bf z)e^{-\bf w\cdot \bf T}\mathds 1 = e^{\bf w\cdot \bf T}Y(a,\bf z)\mathds 1 = e^{\bf w\cdot \bf T}e^{\bf z\cdot \bf T}a.\]
\end{proof}
Setting $\bf w=-\bf z$ gives $e^{\bf z\cdot \bf T}e^{-\bf z\cdot \bf T} = \id_V$.
Reading off the coefficients of the above lemma, we get for all $\bf k,\boldsymbol\ell\geq \bf 0$
\begin{equation}\label{eq:transprod}
    T^{(\bf k)}T^{(\boldsymbol \ell)} = T^{(\boldsymbol \ell)}T^{(\bf k)} = \binom{\bf k+\boldsymbol\ell}{\bf k}T^{(\bf k+\boldsymbol\ell)}.
\end{equation}
\begin{cor}
    If $R$ is a $\Q$-algebra, then $T^{(\bf k)} = \frac{1}{\bf k!}\bf T^{\bf k}$ for $\bf T=(T_1,\ldots,T_n)$, where $T_i = T^{(\bf e_i)}$ for all $i\in [1,n]$.
\end{cor}

The first major identity we provide is skew-symmetry. The proof here is similar to the proof in \cite[Proposition 3.2.2]{ravioli}, using axioms V3, V4, and V5:
\begin{prop}[Skew-symmetry]\label{prop:skewsymmetry}
    For all $a,b\in V$ homogeneous we have
    \begin{equation}
        Y(a,\bf z)b = p(a,b)e^{\bf z\cdot \bf T}Y(b,-\bf z)a.
    \end{equation}
\end{prop}
\begin{proof}
    By V5, let $C^{\bf j}(\bf w) = Y(a,\bf w)_{(\bf j)}Y(b,\bf w)$ be fields such that
    \[[Y(a,\bf z),Y(b,\bf w)] = \sum_{\bf j\geq \bf 0}\partial^{(\bf j)}_{\bf w}\Delta(\bf z,\bf w)C^{\bf j}(\bf w).\]
    Applying the vacuum, we get by V3 and V4
    \begin{align*}
        [Y(a,\bf z),Y(b,\bf w)]\mathds 1 &= Y(a,\bf z)e^{\bf w\cdot \bf T}b - p(a,b)Y(b,\bf w)e^{\bf z\cdot \bf T}a\\
        &= Y(a,\bf z)e^{\bf w\cdot \bf T}b - p(a,b)e^{\bf z\cdot \bf T}Y(b,\bf w-\bf z)a.
    \end{align*}
    Splitting the above equation into parts implies the following four identities:
    \begin{align*}
        Y(a,\bf z)_+e^{\bf w\cdot \bf T}b &= p(a,b)e^{\bf z\cdot \bf T}Y(b,\bf w-\bf z)_+a,\\
        Y(a,\bf z)_-e^{\bf w\cdot \bf T}b &= \sum_{\bf j\geq \bf 0}\partial^{(\bf j)}_{\bf w}\Delta_-(\bf z,\bf w)C^{\bf j}(\bf w)_+\mathds 1,\\
        p(a,b)e^{\bf z\cdot \bf T}Y(b,\bf w-\bf z)_-a &= -\sum_{\bf j\geq \bf 0}\partial^{(\bf j)}_{\bf w}\Delta_+(\bf z,\bf w)C^{\bf j}(\bf w)_+\mathds 1,\\
        0&= \sum_{\bf j\geq \bf 0}\partial^{(\bf j)}_{\bf w}\Delta_-(\bf z,\bf w)C^{\bf j}(\bf w)_-\mathds 1.
    \end{align*}
    Set $\bf w=\bf 0$ in the first identity to get $Y(a,\bf z)_+b = p(a,b)e^{\bf z\cdot \bf T}Y(b,-\bf z)_+a$, which is one half of the skew-symmetry identity.

    Taking $\bf w=\bf 0$ in the second identity gives us for all $\bf k\geq \bf 0$
    \[a_{(\bf k)}b = C_{(-\bf 1)}^{\bf k}\mathds 1,\]
    where $C_{(-\bf 1)}^{\bf k}$ is the coefficient in front of $\Omega^{\bf 0}_{\bf w}$ in $C^{\bf k}(\bf w)$.
    Recall that $\partial^{(\bf k)}_{\bf z}\Delta_-(\bf w,\bf z) = (-1)^{|\bf k|+N}\partial^{(\bf k)}_{\bf w}\Delta_+(\bf z,\bf w)$ by properties (3) and (4) from Lemma \ref{lem:deltaidentities}, so substituting this into the series $Y(b,\bf w-\bf z)_-$ gives us
    \begin{align*}
        &p(a,b)e^{\bf z\cdot \bf T}Y(b,\bf w-\bf z)_-a \\
        % &= p(a,b)\sum_{m,\ell\geq 0}z^\ell (-1)^{|m|+N}\partial_w^{(m)}\Delta_+(\bf z,\bf w)T^{(\ell)} b_{(m)}a\\
        &= p(a,b)\sum_{\bf j,\bf m,\boldsymbol\ell\geq \bf 0}\binom{\boldsymbol\ell}{\bf j}\bf w^{\bf j}(\bf z-\bf w)^{\boldsymbol\ell-\bf j} (-1)^{|\bf m|+N}\partial_{\bf w}^{(\bf m)}\Delta_+(\bf z,\bf w)T^{(\boldsymbol\ell)} b_{(\bf m)}a\\
        % &= p(a,b)\sum_{j,m,\ell\geq 0}\binom{\ell}{j}w^j (-1)^{|m|+N}\partial_w^{(m-\ell+j)}\Delta_+(\bf z,\bf w)T^{(\ell)} b_{(m)}a\\
        &= p(a,b)\sum_{\bf j,\bf k,\boldsymbol\ell\geq \bf 0}\binom{\ell}{\bf j}\bf w^{\bf j} (-1)^{|\bf k|+|\boldsymbol\ell|+N}\partial_{\bf w}^{(\bf k+\bf j)}\Delta_+(\bf z,\bf w)T^{(\boldsymbol\ell)} b_{(\bf k+\boldsymbol\ell)}a.
    \end{align*}
    Now look at the $\bf w^{\bf 0}$ term above, which we can recover on both sides by taking residues. We have for all $\bf k\geq \bf 0$
    \[p(a,b)\sum_{\boldsymbol\ell\geq \bf 0} (-1)^{|\bf k|+|\boldsymbol\ell|+N+1}T^{(\boldsymbol\ell)} b_{(\bf k+\boldsymbol\ell)}a = C^{\bf k}_{(-\bf 1)}\mathds 1 = a_{(\bf k)}b.\]
    Equivalently, under the convention $\Omega^{\bf m}_{-\bf z} = (-1)^{|\bf m|+N+1}\Omega^{\bf m}_{\bf z}$,
    \[p(a,b)e^{\bf z\cdot \bf T}Y(b,-\bf z)_-a = Y(a,\bf z)_-b,\]
    which finishes the proof.
\end{proof}

\begin{cor}\label{cor:trans}
    For all $a\in V$ we have $Y(e^{\bf w\cdot \bf T}a,\bf z) = Y(a,\bf z+\bf w)$.
\end{cor}
\begin{proof}
%    We can use Goddard uniqueness. The $w^k$ component of the above identity is $Y(T^{(\bf k)}a,z) = \partial^{(k)}_zY(a,z)$, and to see this we note that
%    \[Y(T^{(\bf k)}a,z)\mathds 1 = e^{\bf z\cdot \bf T}T^{(\bf k)}a = T^{(\bf k)}e^{\bf z\cdot \bf T}a = \partial^{(k)}_zY(a,z)\mathds 1.\]
     We use skew-symmetry, V4, and Lemma \ref{lem:exp}:
     \begin{align*}
         Y(e^{\bf w\cdot \bf T}a,\bf z)b &= p(a,b)e^{\bf z\cdot \bf T}Y(b,-\bf z)e^{\bf w\cdot \bf T}a\\
         &= p(a,b)e^{\bf z\cdot \bf T}e^{\bf w\cdot \bf T}Y(b,-\bf z-\bf w)a = Y(a,\bf z+\bf w)b.
     \end{align*}
\end{proof}
Reading off both sides of this equation, we have for all $\bf k,\boldsymbol\ell\geq \bf 0$
\begin{equation}
(T^{(\bf k)}a)_{(-\bf 1-\boldsymbol\ell)} = \binom{\boldsymbol\ell+\bf k}{\bf k}a_{(-1-\boldsymbol\ell-\bf k)},\qquad (T^{(\bf k)}a)_{(\boldsymbol\ell)} = (-1)^{|\bf k|}\binom{\boldsymbol\ell}{\bf k}a_{(\boldsymbol\ell-\bf k)}.
\end{equation}

Now we prove an analog of Goddard uniqueness.
\begin{prop}[Goddard uniqueness]\label{prop:Goddard}
    Let $A(\bf z)$ be a homogeneous field on $V$ that is mutually local with $Y(b,w)$ for all $b\in V$. If there exists a homogeneous $a\in V$ such that
    \[A(\bf z)\mathds 1 = Y(a,\bf z)\mathds 1,\]
    then $A(\bf z) = Y(a,\bf z)$.
\end{prop}
\begin{proof}
    The equality $a_{(-\bf 1-\bf k)}\mathds 1 = A_{(-\bf 1-\bf k)}\mathds 1$ implies that $A_{(-\bf 1-\bf k)}$ is homogeneous of degree $p(A) = p(a)$ for all $\bf k\geq \bf 0$.
    By locality, there exists a $\bf K\geq \bf 0$ such that $(z_i-w_i)^{K_i}[A(\bf z),Y(b,\bf w)]\mathds 1=0$ and $(z_i-w_i)^{K_i}[Y(a,\bf z),Y(b,\bf w)]\mathds 1=0$ for all $i$. Taking nonsingular parts, we have
    \begin{align*}
        (z_i-w_i)^{K_i}A(\bf z)_+Y(b,\bf w)_+\mathds 1 &= p(A,b)(z_i-w_i)^{K_i}Y(b,\bf w)A(\bf z)_+\mathds 1\\ 
        &= p(A,b)(z_i-w_i)^{K_i}Y(b,\bf w)_+Y(a,\bf z)_+\mathds 1 \\
        &= (z_i-w_i)^{K_i}Y(a,\bf z)_+Y(b,\bf w)_+\mathds 1.
    \end{align*}
    Now set $w=0$ on both sides to get $z_i^{K_i}A(\bf z)_+b = z_i^{K_i}Y(a,\bf z)_+b$, which implies that $A(\bf z)_+b = Y(a,\bf z)_+b$ for each $b$, hence $A(\bf z)_+ = Y(a,\bf z)_+$.

    On the other hand, by locality we have an expansion
    \[[A(\bf z)-Y(a,\bf z),Y(b,\bf w)] = \sum_{\bf m\geq \bf 0}\partial^{(\bf m)}_{\bf w}\Delta(\bf z,\bf w)C^{\bf m}(\bf w).\]
    Applying the vacuum and using the fact that $A(z)_+ = Y(a,z)_+$, we get
    \[(A(\bf z)_- -Y(a,\bf z)_-)Y(b,\bf w)\mathds 1 = \sum_{\bf m\geq \bf 0}\partial^{(\bf m)}_{\bf w}\Delta(\bf z,\bf w)C^{\bf m}(\bf w)\mathds 1.\]
    The left-hand side does not contain any terms with a $\Omega_{\bf w}^{\bf m}$ factor, so the right-hand side must be 0. Thus we have
    \[(A(\bf z)_- -Y(a,\bf z)_-)Y(b,\bf w)\mathds 1 = 0.\]
    Setting $\bf w=\bf 0$ gives us $A(\bf z)_-b = Y(b,\bf w)_-b$ for all $b\in V$.
\end{proof}
We will only need this result when proving an analog of the reconstruction theorem (Theorem \ref{thm:reconstruction}).

Now we consider different associativity relations for CVAs. For these identities, we need to make sense of the composition $Y(Y(a,\bf z)b,\bf w)$ for $a,b\in V$. This expression is a slight abuse of notation since $Y$ as a map takes an element of $V$ as input, not an element of $V\langle\!\langle z\rangle\!\rangle$, which is what $Y(a,\bf z)b$ is. In the theory of VAs, the discussion is straightforward due to a lack of cohomological signs, but here it is a bit more subtle.
\begin{defn}
    For $a,b\in V$ we define the composition $Y(Y(a,\bf z)b,\bf w)$ as follows:
    \begin{equation}
        Y(Y(a,\bf z)b,\bf w) \eqdef \sum_{\bf k\geq \bf 0}\bf z^{\bf k}Y(a_{(-\bf 1-\bf k)}b,\bf w) + \Omega_{\bf z}^{\bf k}Y(a_{(\bf k)}b,\bf w).
    \end{equation}
\end{defn}
This definition agrees with skew-symmetry: \[Y(Y(a,\bf z)b,\bf w)c = p(b,c)p(a,c)e^{\bf w\cdot \bf T}Y(c,-\bf w)Y(a,\bf z)b.\]
We will see this definition employed when formulating versions of weak associativity and the iterate and Jacobi identities.
\begin{prop}[Weak associativity]\label{prop:weakassoc}
    For all $a,c\in V$, the series
    \begin{equation}
        C(\bf z,\bf w) \eqdef Y(a,\bf z-\bf w)Y(b,-\bf w)c - Y(Y(a,\bf z)b,-\bf w)c
    \end{equation}
    is local for the pair of variables $(\bf z,\bf w)$ with the number of terms in the local expansion uniformly bounded for all $\bf b\in V$.
\end{prop}
\begin{proof}
    It suffices to check when $a,b,c\in V$ are homogeneous. Expanding the left-hand side using V4 and skew-symmetry, we have
    \begin{align*}
        &Y(a,\bf z-\bf w)Y(b,-\bf w)c - Y(Y(a,\bf z)b,-\bf w)c \\
        &= e^{-\bf w\cdot \bf T}Y(a,\bf z)e^{\bf w\cdot \bf T}Y(b,-\bf w)c- p(b,c)p(a,c)e^{-\bf w\cdot \bf T}Y(c,\bf w)Y(a,\bf z)b\\
        &= p(b,c)e^{-\bf w\cdot \bf T}[Y(a,\bf z),Y(c,\bf w)]b.
    \end{align*}
    The bracket $[Y(a,\bf z),Y(c,\bf w)]$ local for all $a,c\in V$ by V5. We may multiply by the power series $p(a,b)e^{-\bf w\cdot \bf T}$, and the resulting series is still local.
\end{proof}
We note that weak associativity is equivalent to the associativity property in \cite[Theorem 3.3.1]{ravioli}, which is a more computationally tractable identity to work with:
\begin{prop}[Associativity]\label{prop:assoc}
    For $a,b\in V$ homogeneous, we have that for all $c\in V$ the series
    \begin{align*}
        Y(a,\bf z)Y(b,\bf w)c&\in V\langle\!\langle \bf z\rangle\!\rangle\langle\!\langle \bf w\rangle\!\rangle,\\
        p(a,b)Y(b,\bf w)Y(a,\bf z)c&\in V\langle\!\langle \bf w\rangle\!\rangle\langle\!\langle \bf z\rangle\!\rangle,\\
        Y(Y(a,\bf z-\bf w)b,\bf w)c&\in V\langle\!\langle \bf w\rangle\!\rangle\langle\!\langle \bf z-\bf w\rangle\!\rangle
    \end{align*}
    are expansions, in their respective domains, of the same element of $\cal K^{\bf z,\bf w,\bf z-\bf w}(V)$.
\end{prop}
\begin{proof}
    The first two series being expansions of the same element of $\cal K^{\bf z,\bf w,\bf z-\bf w}(V)$ follows from Proposition \ref{prop:localanalytic}.
    By the same argument, we note that weak associativity implies that $Y(a,\bf x-\bf y)Y(b,-\bf y)c$ and $Y(Y(a,\bf x)b,-\bf y)c$ are images of the same element of $\cal K^{\bf x,\bf y,\bf x-\bf y}(V)$ in the codomains $V\langle\!\langle \bf x\rangle\!\rangle\langle\!\langle \bf y\rangle\!\rangle$ and $V\langle\!\langle \bf y\rangle\!\rangle\langle\!\langle \bf x\rangle\!\rangle$ respectively. Now make the formal variable substitution $(\bf x,\bf y)=(\bf z-\bf w,-\bf w)$ in $\cal K^{\bf x,\bf y,\bf x-\bf y}$, so 
    \[Y(a,\bf z)Y(b,\bf w)c \in V\langle\!\langle \bf z-\bf w\rangle\!\rangle\langle\!\langle -\bf w\rangle\!\rangle=V\langle\!\langle \bf z\rangle\!\rangle\langle\!\langle \bf w\rangle\!\rangle\]
    and similarly
    \[Y(Y(a,\bf z-\bf w)b,\bf w)c \in V\langle\!\langle -\bf w\rangle\!\rangle\langle\!\langle \bf z-\bf w\rangle\!\rangle = V\langle\!\langle \bf w\rangle\!\rangle\langle\!\langle \bf z-\bf w\rangle\!\rangle.\]
\end{proof}
Using the associativity relation for VAs, we can prove many convenient identities.
For example, the associativity property as stated above can be used to show the iterate formula \cite[Remark 3.1.12]{LepowskyLi}. Indeed the same approach works here.
\begin{prop}[Iterate/associator formula]\label{prop:iterate}
    For all $a,b\in V$ homogeneous and $\bf k\geq \bf 0$ we have
    \begin{align}
        \label{eq:kprod1} Y(a_{(\bf k)}b,\bf w) &= Y(a,\bf w)_{(\bf k)}Y(b,\bf w) \eqdef \Res_{\bf z}\left((\bf z-\bf w)^{\bf k}[Y(a,\bf z),Y(b,\bf w)]\right),\\
        \label{eq:kprod2} Y(a_{(-\bf 1-\bf k)}b,\bf w) &= Y(a,\bf w)_{(-\bf 1-\bf k)}Y(b,\bf w) \eqdef \ \normord{\partial_{\bf w}^{(\bf k)}Y(a,\bf w)Y(b,\bf w)}.
    \end{align}
    Equivalently, we have the series identity
    \begin{equation}
        Y(Y(a,\bf z)b,\bf w) = Y(a,\bf z+\bf w)Y(b,\bf w) - p(a,b)Y(b,\bf w)\left(Y(a,\bf z+\bf w)_--Y(a,\bf w+\bf z)_-\right).
    \end{equation}
\end{prop}
\begin{proof}
%    First, we show the two statements are equivalent.
%    \begin{align*}
%        Y(Y(a,z)_+b,w) &= \sum_{k\geq 0}z^kY(a_{(-1-k)}b,w)\\
%        &= \sum_{k\geq 0}z^k\normord{\partial^{(k)}_wY(a,w)Y(b,w)} = \normord{Y(a,w+z)Y(b,w)}.
%    \end{align*}
%    Similarly, we have
%    \begin{align*}
%        Y(Y(a,z)_-b, w) 
%        % &= \sum_{k\geq 0}\Omega^{\bf k}_{\bf z}Y(a_{(k)}b,w)\\
%        &= \sum_{k\geq 0}\Omega^{\bf k}_{\bf z}\Res_x\left((x-w)^k[Y(a,x),Y(b,w)]\right)\\
%        % &= \sum_{k\geq 0}\Omega^{\bf k}_{\bf z}\sum_{\ell\geq 0}\binom{k}{\ell}(-w)^{k-\ell}[a_{(\ell)},Y(b,w)]\\
%        % &= \sum_{k,\ell\geq 0}\binom{k+\ell}{\ell}(-w)^{k}\Omega^{k+\ell}_z[a_{(\ell)},Y(b,w)]\\
%        &= \sum_{\ell\geq 0}\partial_{-w}^{(\ell)}\Delta_-(z+w)[a_{(\ell)},Y(b,w)]\\
%        &= [Y(a,z+w)_-,Y(b,w)].
%    \end{align*}
%    Combining the two gives us the desired series identity, giving us the stated equivalence.
    We use associativity for CVAs. In $V\langle\!\langle \bf w\rangle\!\rangle\langle\!\langle \bf z-\bf w\rangle\!\rangle$, we have
    \begin{equation*}
        Y(Y(a,\bf z-\bf w)b,\bf w)c = \sum_{\bf k\geq \bf 0}(\bf z-\bf w)^{\bf k}Y(a_{(-\bf 1-\bf k)}b,\bf w) + \Omega_{\bf z-\bf w}^{\bf k}Y(a_{(\bf k)}b,\bf w),
    \end{equation*}
    and in $V\langle\!\langle \bf z\rangle\!\rangle\langle\!\langle \bf w\rangle\!\rangle$ we have the operator product expansion
    \begin{equation*}
        Y(a,\bf z)Y(b,\bf w)c =\, \normord{Y(a,\bf z)Y(b,\bf w)} + \sum_{\bf k\geq \bf 0}\partial^{(\bf k)}_{\bf w}\Delta_-(\bf z,\bf w) Y(a,\bf w)_{(\bf k)}Y(b,\bf w).
    \end{equation*}
    Now we take the preimage of $Y(a,\bf z)Y(b,\bf w)c$ in $\cal K^{\bf z,\bf w,\bf z-\bf w}$. In the region $V\langle\!\langle \bf w\rangle\!\rangle\langle\!\langle \bf z-\bf w\rangle\!\rangle$, the above equation translates to
    \begin{align*}
        Y(a,\bf z)Y(b,\bf w)c &=\, \normord{Y(a,\bf w+(\bf z-\bf w))Y(b,\bf w)} + \sum_{\bf k\geq \bf 0}\Omega_{\bf z-\bf w}^{\bf k} Y(a,\bf w)_{(\bf k)}Y(b,\bf w)\\
        &=\sum_{\bf k\geq \bf 0}(\bf z-\bf w)^k\normord{\partial^{(\bf k)}_{\bf w}Y(a,\bf w)Y(b,\bf w)} + \sum_{\bf k\geq \bf 0}\Omega_{\bf z-\bf w}^{\bf k} Y(a,\bf w)_{(\bf k)}Y(b,\bf w).
    \end{align*}
    Comparing this to $Y(Y(a,\bf z-\bf w)b,\bf w)c$ gives us the desired identities.
\end{proof}
The iterate formula allows us to prove many useful identities for CVAs, such as the following:
\begin{cor}[Commutator formula]\label{cor:commutator formula}
% The following hold for all $a,b\in V$:
% \begin{enumerate}
    We have the following identity for $a,b\in V$:
    \begin{equation}
    \begin{aligned}
        [Y(a,\bf z),Y(b,\bf w)] &= \sum_{\bf 0\leq \bf k\leq \bf K}\partial^{(\bf k)}_{\bf w}\Delta(\bf z,\bf w)Y(a_{(\bf k)}b,\bf w)\\
        &= Y(Y(a,\bf z-\bf w)_-b-Y(a,-\bf w+\bf z)_-b,\bf w),
    \end{aligned}
    \end{equation}
    where $\bf K$ is the bound from axiom V1.
    % \item We have the following operator product expansions for $a,b\in V$ homogeneous:
    % \begin{align}
    %     Y(a,z)Y(b,w) &= \sum_{j\geq 0}\partial^{(\bf j)}_{\bf w}\Delta_-(\bf z,\bf w)Y(a_{(j)}b,w) + \normord{Y(a,z)Y(b,w)},\\
    %     p(a,b)Y(b,w)Y(a,z) &= -\sum_{j\geq 0}\partial^{(\bf j)}_{\bf w}\Delta_+(\bf z,\bf w)Y(a_{(j)}b,w) + \normord{Y(a,z)Y(b,w)}.
    % \end{align}
    % \item For all $a,b,c\in V$ we have
    % \begin{equation}
    %     \begin{aligned}
    %         &[Y(a,z),Y(b_{(-1)}c,w)] \\
    %         &= \sum_{k\geq 0}\partial^{(k)}_w\Delta(\bf z,\bf w) Y(a_{(k)}(b_{(-1)}c),w)\\
    %         &= \sum_{k\geq 0}\partial^{(k)}_w\Delta(\bf z,\bf w)\left(Y((a_{(k)}b)_{(-1)}c,w) + (-1)^{p(b)(p(a)+N)}Y(b_{(-1)}(a_{(k)}c),w)\right).
    %     \end{aligned}
    % \end{equation}
    % Hence $a_{(k)}(b_{(-1)}c) = (a_{(k)}b)_{(-1)}c + (-1)^{p(b)(p(a)+N)}b_{(-1)}(a_{(k)}c)$ for all $k\geq 0$.
    % \end{enumerate}
\end{cor}

Now we use locality and the iterate formula to prove our analog of the Jacobi identity, which first appeared in \cite[Theorem 8.8.9]{FLM}. This identity is formulated in terms of the three-variable $\Delta$-functions $\Delta(\bf x-\bf y,\bf z) = e^{-\bf y\cdot \partial_{\bf x}}\Delta(\bf x,\bf z)$.
\begin{thm}[Jacobi identity]\label{thm:Jacobi}
    For all $a,b\in V$ homogeneous and $c\in V$ we have
    \begin{equation}
    \begin{aligned}
        & \Delta(\bf x-\bf y,\bf z)Y(a,\bf x)Y(b,\bf y)c - p(a,b)\Delta(-\bf y+\bf x,\bf z)Y(b,\bf y)Y(a,\bf x)c\\
        & =\Delta(\bf x-\bf z,\bf y)Y(Y(a,\bf z)b,\bf y)c.
    \end{aligned}
    \end{equation}
\end{thm}
\begin{proof}
    The left-hand side of the Jacobi identity can be expanded in $z$ as
    \begin{align*}
        &\Delta(x-y,z)Y(a,\bf x)Y(b,\bf y) - p(a,b)\Delta(-\bf y+\bf x,\bf z)Y(b,\bf y)Y(a,\bf x)\\
        &=\sum_{\bf k\geq \bf 0}(-1)^N\Omega_{\bf z}^{\bf k}(\bf x-\bf y)^{\bf k}[Y(a,\bf x),Y(b,\bf y)]\\
        &\qquad +\bf z^{\bf k}\partial_{\bf y}^{(\bf k)}\Delta_-(\bf x,\bf y)Y(a,\bf x)Y(b,\bf y) + p(a,b)\bf z^{\bf k}\partial_{\bf y}^{(\bf k)}\Delta_+(\bf x,\bf y)Y(b,\bf y)Y(a,\bf x).
    \end{align*}
    By the commutator formula we have
    \begin{equation}
        (\bf x-\bf y)^{\bf k}[Y(a,\bf x),Y(b,\bf y)] = \sum_{\bf j\geq \bf 0}\partial^{(\bf j)}_{\bf y}\Delta(\bf x,\bf y)Y(a_{(\bf k+\bf j)}b,\bf y).
    \end{equation}
    A straightforward expansion shows that
    \begin{equation}
    \begin{aligned}
        &\partial_{\bf y}^{(\bf k)}\Delta_-(\bf x,\bf y)Y(a,\bf x)Y(b,\bf y) + p(a,b)\partial_{\bf y}^{(\bf k)}\Delta_+(\bf x,\bf y)Y(b,\bf y)Y(a,\bf x)\\ 
        &= \sum_{\bf 0\leq \bf j\leq \bf k}\partial_{\bf y}^{(\bf j)}\Delta(\bf x,\bf y)Y(a,\bf y)_{(-\bf 1-\bf k+\bf j)}Y(b,\bf y)\\
        &= \sum_{\bf 0\leq \bf j\leq \bf k}\partial_{\bf y}^{(\bf j)}\Delta(\bf x,\bf y)Y(a_{(-\bf 1-\bf k+\bf j)}b,\bf y),
    \end{aligned}
    \end{equation}
    and the sum on the right hand side is in fact finite. Combining these two equations gives us the Jacobi identity.
    % The right-hand side can be written as
    % \begin{align*}
    %     &\Delta(x-z,y)Y(Y(a,z)b,y)\\
    %     &= \sum_{j,k\geq 0}\partial_y^{(j)}\Delta(x,y)\left(z^k Y(a_{(-1-k+j)}b,y) + \Omega^{\bf k}_{\bf z}Y(a_{(k+j)}b,y)\right)\\
    %     &= \sum_{j,k\geq 0}z^k\partial_y^{(j)}\Delta(x,y) Y(a_{(-1-k+j)}b,y) + (-1)^N\Omega^{\bf k}_{\bf z}\partial_y^{(j)}\Delta(x,y)Y(a_{(k+j)}b,y).
    % \end{align*}
    % Comparing both sides and using the first two identities we showed establishes the Jacobi identity.
\end{proof}

\subsection{Equivalence of various axioms for vertex algebras}\label{subsec:equivdefn}
Here we outline some equivalent axiomatic presentations so that one may arrive at a CVA in different ways.

\subsubsection{The case of characteristic 0}
We start by showing how to replace axioms V2, V3, and V4 with an alternative set of axioms when $R$ is a $\Q$-algebra, similar to \cite{kac1998vertex} and \cite{frenkel2004vertex}.
\begin{prop}\label{prop:Qalg}
    Let $R$ be a $\Q$-algebra. Then a cohomological vertex algebra $V$ is equivalent to the data $(V,Y,\bf T,\mathds 1)$, where $\bf T = (T_1,\ldots,T_n)\subset \End(V)$ is a collection of degree 0 endomorphisms for all $i$ such that the following hold:
    \begin{enumerate}
        \item For all $a,b\in V$, we have $Y(a,\bf z)b\in V\langle\!\langle \bf z\rangle\!\rangle$.
        \item For all $a\in V$ the series $Y(a,\bf z)\mathds 1$ is a power series such that the constant term is $a$. That is, we have $Y(a,\bf z)\mathds 1|_{\bf z=\bf 0} = a$.
        \item For all $i$ we have $[T_i,Y(a,\bf z)] = \partial_{z_i}Y(a,\bf z)$.
        \item For all $i$ we have $T_i\mathds 1 = 0$.
        \item $Y(a,\bf z)$ and $Y(b,\bf w)$ are mutually local for all $a,b\in V$.
    \end{enumerate}
\end{prop}
\begin{proof}
    
    Axioms (1) and (5) are identical to V1 and V5 respectively.
    Note that (3) implies that
    \begin{align*}
        [[T_i,T_j],Y(a,\bf z)] &= [T_i,[T_j,Y(a,\bf z)]] - [T_j,[T_i,Y(a,\bf z)]]\\
        &= \partial_{z_i}\partial_{z_j}Y(a,\bf z) - \partial_{z_j}\partial_{z_i}Y(a,\bf z) = 0.
    \end{align*}
    Apply both sides to $\mathds 1$ and set $z=0$ to get that $[T_i,T_j]a = 0$ by (4) and (2).
    Therefore $[T_i,T_j]=0$.

    It is straightforward to see that $Y(a,\bf z)\mathds 1 = e^{\bf z\cdot \bf T}a$ if and only if $Y(a,\bf z)\mathds 1|_{z=0}=a$ and $\partial_{z_i}Y(a,\bf z)\mathds 1 = T_iY(a,\bf z)\mathds 1$. The initial condition follows from (2), and the derivative identity follows from (3). This shows V3, and we define $T^{(\bf k)} = \frac{1}{\bf k!}\bf T^{\bf k}$.
    
    Note that $[T_i,T_j]=0$ for all $i,j$, so V4 is equivalent to
    \[e^{\bf w\cdot \bf T}Y(a,\bf z)e^{-\bf w\cdot \bf T}=\sum_{k\geq 0}\frac{\bf w^{\bf k}}{\bf k!}\prod_{i=1}^n(\ad_{T_i})^{k_i}Y(a,\bf z) = \sum_{\bf k\geq \bf 0}\bf w^{\bf k}\partial^{(\bf k)}_{\bf z}Y(a,\bf z),\]
    hence V4 follows from the first property of (3).

    We note that the proof of skew-symmetry only uses V3, V4, and V5. Thus we have
    \[Y(\mathds 1,\bf z)a = e^{\bf z\cdot \bf T}Y(a,-\bf z)\mathds 1 = e^{\bf z\cdot \bf T}e^{-\bf z\cdot \bf T}a=a.\]
    This shows V2, implying that the set of axioms (1)-(5) imply V1-V5. The other direction is clear.
\end{proof}

\subsubsection{Equivalent presentations of the Jacobi identity}
Assuming axiom V1, the Jacobi identity is equivalent to V5 and the iterate formula. Moreover, we obtain the following:
\begin{prop}
    The Jacobi identity for cohomological vertex algebras is equivalent to weak commutativity (V5) together with weak associativity (\ref{prop:weakassoc}), or equivalently the associativity relation (\ref{prop:assoc}).
\end{prop}
\begin{proof}
    We just need to show that the iterate formula implies weak associativity. 
    The iterate formula reads as
    \begin{align*}
        & Y(a,\bf z-\bf w)Y(b,-\bf w) - Y(Y(a,\bf z)b,-\bf w) \\
        & =  p(a,b)Y(b,-\bf w)\left(Y(a,\bf z-\bf w)-Y(a,-\bf w+\bf z)\right)\\
        & =  (-1)^{p(b)(p(a)+N)}\sum_{\bf k\geq \bf 0}\partial^{(\bf k)}_{\bf w}\Delta(\bf z,\bf w)Y(b,-\bf w)a_{(\bf k)}.
    \end{align*}
    Applying an element $c\in V$, the number of terms on the right-hand side depends only on $a$ and $c$, so weak associativity follows.
    % Applying both sides to $\mathds 1$, by V3 we get
    % \[Y(a,\bf z-\bf w)e^{-\bf w\cdot \bf T}b - e^{-\bf w\cdot \bf T}Y(a,z)b =  p(a,b)Y(b,-w)\left(e^{(\bf z-\bf w)\cdot \bf T}-e^{(-\bf w+\bf z)\cdot \bf T}\right)a = 0.\]
    % Thus we have V4 since skew-symmetry together with V2 and V3 imply $e^{\bf w\cdot \bf T}e^{-\bf w\cdot \bf T}=\id_V$.
    % Now note that
    % \[\Delta_-(-w,-z) = \sum_{k\geq 0}(-z)^k\Omega_{-w}^k = -\Delta_+(\bf z,\bf w).\]
    % Hence we have
    % \begin{align*}
    %     Y(a,\bf z-\bf w)-Y(a,-\bf w+\bf z) &= \sum_{k\geq 0}(\partial^{(k)}_w\Delta_-(\bf z,\bf w) - \partial^{(k)}_{-z}\Delta_-(-w,-z) )a_{(k)}\\
    %     &= \sum_{k\geq 0}(\partial^{(k)}_w\Delta_-(\bf z,\bf w) + \partial^{(k)}_{w}\Delta_+(\bf z,\bf w) )a_{(k)}\\
    %     &= \sum_{k\geq 0}\Delta(\bf z,\bf w)a_{(k)}.
    % \end{align*}
    % Commuting $Y(b,-w)$ past the $\Delta(\bf z,\bf w)$ shows that for all $c\in V$ the expression 
    % \[Y(a,\bf z-\bf w)Y(b,-w)c - Y(Y(a,z)b,-w)c\]
    % is local in $\cal K^{z,w}_{\mathrm{dist}}\otimes_RV$, which shows weak associativity. 
    % Applying the proof of weak associativity (using skew-symmetry and V4) again we have that
    % \[Y(a,\bf z-\bf w)Y(b,-w)c - Y(Y(a,z)b,-w)c = p(b,c)e^{-\bf w\cdot \bf T}[Y(a,z),Y(c,w)]b.\]
    % From this we use the iterate formula again to conclude that
    % \begin{align*}
    %     p(b,c)e^{-\bf w\cdot \bf T}[Y(a,z),Y(c,w)]b &= p(a,b)Y(b,-w)\left(Y(a,\bf z-\bf w)-Y(a,-\bf w+\bf z)\right)c,
    % \end{align*}
    % and from $e^{\bf w\cdot \bf T}e^{-\bf w\cdot \bf T}=\id_V$ we obtain the commutator formula after applying skew-symmetry, which by V1 implies V5.
\end{proof}

\subsubsection{Replacing translation covariance and weak commutativity}
Now we discuss how to replace axioms V4 and V5 of a CVA.
To start, we provide two lemmas.
\begin{lem}
    Skew-symmetry, V2, and V3 together imply that $e^{\bf z\cdot \bf T}e^{-\bf z\cdot \bf T}=\id_V$.
\end{lem}
\begin{proof}
    Take $a=\mathds 1$, and apply any $b\in V$:
    \[Y(\mathds 1,\bf z)b = b = e^{\bf z\cdot \bf T}Y(b,-\bf z)\mathds 1 = e^{\bf z\cdot \bf T}e^{-\bf z\cdot \bf T}b.\]
\end{proof}
\begin{lem}
    Weak associativity, V3, and $e^{\bf z\cdot \bf T}e^{-\bf z\cdot \bf T}=\id_V$ together imply V4.
\end{lem}
\begin{proof}
    Take $c=\mathds 1$. By weak associativity, the series
    \begin{align*}
        C(\bf z,\bf w) &= Y(a,\bf z-\bf w)Y(b,-\bf w)\mathds 1 - Y(Y(a,\bf z)b,-\bf w)\mathds 1\\
        & = Y(a,\bf z-\bf w)e^{-\bf w\cdot \bf T}b - e^{-\bf w\cdot \bf T}Y(a,\bf z)b.
    \end{align*}
    By weak associativity this expression is local, meaning that it can be written as a sum of terms of the form $\partial^{(\bf k)}_{\bf w}\Delta(\bf z,\bf w)g^{\bf j}(\bf w)$ for various series $g^{\bf j}(\bf w)$. Note that $g^{\bf j}(\bf w) = \Res_{\bf z}\left((\bf z-\bf w)^{\bf j}C(\bf z,\bf w)\right)$. Since $C(\bf z,\bf w)$ is a power series in $w$, these residues are all 0. This implies that $C(\bf z,\bf w) = 0$, which shows V4.
\end{proof}
We now provide our first reformulation of a cohomological vertex algebra:
\begin{prop}
    The axioms V4 and V5 of a cohomological vertex algebra may be replaced by skew-symmetry and weak associativity.
\end{prop}
\begin{proof}
    Skew-symmetry implies $e^{\bf z\cdot \bf T}e^{-\bf z\cdot \bf T}=\id_V$, which together with weak associativity implies V4. We can now carry out the proof of weak associativity to obtain
    \[Y(a,\bf z-\bf w)Y(b,-\bf w)c - Y(Y(a,\bf z)b,-\bf w)c = p(b,c)e^{-\bf w\cdot \bf T}[Y(a,\bf z),Y(c,\bf w)]b.\]
    Both sides of the equation are local by weak associativity, and since applying $e^{\bf w\cdot \bf T}$ preserves locality, the bracket $[Y(a,\bf z),Y(c,\bf w)]b$ is local for all $b$ (with the locality independent of $b$). Therefore $[Y(a,\bf z),Y(c,\bf w)]$ is local.
\end{proof}
The most common presentation of the axioms of a vertex algebra in the literature uses the Jacobi identity, which we show may be done here as well:
\begin{prop}
    The axioms V4 and V5 of a cohomological vertex algebra may be replaced by the Jacobi identity.
\end{prop}
\begin{proof}
    The Jacobi identity is equivalent to V5 and the iterate formula, and by the proof of the previous proposition we can apply the iterate formula to $\mathds 1$ to get
    \[Y(a,\bf z-\bf w)e^{-\bf w\cdot \bf T}b - e^{-\bf w\cdot \bf T}Y(a,\bf z)b = 0.\]
    Now we just need to show $e^{\bf z\cdot \bf T}e^{-\bf z\cdot \bf T} = \id_V$ to get V4. To see this, set $b=\mathds 1$ in the iterate formula. Noting that $p(a,\mathds 1) = 1$ for all $a\in V$ and canceling on both sides gives us the identity
    $Y(e^{\bf z\cdot \bf T}a,\bf w) =  Y(a,\bf w+\bf z)$. Now apply both sides to $\mathds 1$ to get $e^{\bf w\cdot \bf T}e^{\bf z\cdot \bf T}=e^{(\bf w+\bf z)\cdot \bf T}$, and setting $\bf w=-\bf z$ gives us $e^{\bf z\cdot \bf T}e^{-\bf z\cdot \bf T}=\id_V$.
\end{proof}

\subsubsection{Vertex algebras in terms of correlation functions}
We can generalize the graded algebra $\cal K^{\bf z,\bf w,\bf z-\bf w}(V)$ by considering a configuration space in $d$ many sets of variables:
\begin{equation}
    H^\bullet \Conf^\circ_d(V) \eqdef V[\![\bf z^1,\ldots,\bf z^d]\!][\Omega^{\bf m}_{\bf z^k},\Omega^{\bf m}_{\bf z^k-\bf z^l}]_{1\leq k<l\leq d,\, \bf m\geq \bf 0}.
\end{equation}
Here the degree $N$ forms are subject to similar conditions to $\cal K^{\bf z,\bf w,\bf z-\bf w}$. More explicitly, we require that for all $i,j,k$
\[\Omega^{\bf 0}_{\bf x^i-\bf x^j}\Omega^{\bf 0}_{\bf x^j-\bf x^k} + \Omega^{\bf 0}_{\bf x^j-\bf x^k}\Omega^{\bf 0}_{\bf x^k-\bf x^i} + \Omega^{\bf 0}_{\bf x^k-\bf x^i}\Omega^{\bf 0}_{\bf x^i-\bf x^j} = 0.\]
If we set $(\bf x^i,\bf x^j,\bf x^k) = (\bf z,\bf w,\bf 0)$, then
\[\Omega^{\bf 0}_{\bf z-\bf w}\Omega^{\bf 0}_{\bf w} + \Omega^{\bf 0}_{\bf w}\Omega^{\bf 0}_{-\bf z} + \Omega^{\bf 0}_{-\bf z}\Omega^{\bf 0}_{\bf z-\bf w} = (-1)^N((\Omega_{\bf w}^{\bf 0}-\Omega_{\bf z}^{\bf 0})\Omega_{\bf z-\bf w}^{\bf 0} - \Omega_{\bf w}^{\bf 0}\Omega_{\bf z}^{\bf 0}) = 0,\]
which is the same as (3) for $\cal K^{\bf z,\bf w,\bf z-\bf w}$.
This is a version of the Arnold relations.

We may also consider a version of the above, except we omit the forms $\Omega^m_{z^k}$:
\begin{equation}
    H^\bullet \Conf_d(V) \eqdef V[\![\bf z^1,\ldots,\bf z^d]\!][\Omega^{\bf m}_{\bf z^k-\bf z^l}]_{1\leq k< l\leq d,\, \bf m\geq \bf 0}.
\end{equation}
We consider our analog of $d$-point correlation functions, an important object of study in 2D conformal field theories. Before we present the proposition we give a small piece of notation. Let $a^1,\ldots,a^d\in V$ be homogeneous elements. Given $\sigma\in S_d$, decompose it into a composition of pairwise swaps, say $\sigma = (k_1l_1)\circ\ldots\circ (k_rl_r)$ for $k^i,l^i\in [1,d]$. Define
$\chi(\sigma;a^1,\ldots,a^d) = \prod_{i=1}^rp(a^{k_i},a^{l_i})$. It is straightforward to see that this formula is independent of the decomposition of $\sigma$.

One of the primary objects of study in 2D CFTs is the collection of correlation functions, which by the Wightman reconstruction theorem \cite[Theorem 8.18]{SchottenhauerCFT} allows us to recover all information about the fields underlying the theory. There is a similar correspondence for VAs, and we show an analog of the result here for CVAs.
\begin{thm}\label{thm:CVAcorrfuns}
    A $(n,N)$-CVA is equivalent to the data $(V,\mu,e^{\bf z\cdot \bf T},\mathds 1)$, where $V=\bigoplus_{r\in \Z}V^r$ is a $\Z$-graded $R$-module, $\mathds 1\in V^0$, $e^{\bf z\cdot \bf T} = \sum_{k\geq 0}\bf z^{\bf k}T^{(\bf k)}$ is an endomorphism series, and $\mu=(\mu_d)_{d=0}^\infty$ is a collection of degree 0 $R$-linear maps
    \[\mu_d \colon V^{\otimes d} \to H^\bullet \Conf_d(V),\]
    where we operate under the convention that $V^{\otimes 0}=R$. This data satisfies the following axioms:
    \begin{enumerate}
        \item For all $r\in R$ we have $\mu_0(r) = r\mathds 1$.
        \item For all $a\in V$ we have $\mu_1(a)(\bf z) = e^{\bf z\cdot \bf T}a$. In particular, $\mu_1(a)(\bf 0) = a$.
        \item $\mu_d$ is $S_d$-invariant; that is for all $d\geq 1$, $a^1,\ldots,a^d\in V$ homogeneous, and $\sigma\in S_d$ we have
        \begin{equation}
            \mu_d(a^1,\ldots,a^d)(\bf z^1,\ldots,\bf z^d) = \chi(\sigma;a^1,\ldots,a^d)\mu_d(a^{\sigma(1)},\ldots,a^{\sigma(d)})(\bf z^{\sigma(1)},\ldots,\bf z^{\sigma(d)}).
        \end{equation}
        \item For all $d_2\geq 0$, $1\leq i\leq d_1$, and $a^j,b^j\in V$ we have
        \begin{equation}
        \begin{aligned}
            &\mu_{d_1}(a^1,\ldots,a^{i-1},\mu_{d_2}(b^1,\ldots,b^{d_2})(\bf z^1,\ldots,\bf z^{d_2}),a^{i+1},\ldots,a^{d_1})(\bf w^1,\ldots,\bf w^{d_1})\\
            &=i_{\bf z^1,\bf w^i}\circ\ldots\circ i_{\bf z^{d_2},\bf w^i}\mu_{d_1+d_2-1}(a^1,\ldots,a^{i-1},b^1,\ldots,b^{d_2},a^{i+1},\ldots,a^{d_1})\\
            &\qquad\qquad\qquad (\bf w^1,\ldots,\bf w^{i-1},\bf z^1+\bf w^i,\ldots,\bf z^{d_2}+\bf w^i,\bf w^{i+1},\ldots,\bf w^{d_1})
        \end{aligned}
        \end{equation}
    \end{enumerate}
    The correspondence is given by the relation
    \begin{equation}
        i_{\bf z^1,\ldots,\bf z^d}\mu_d(a^1,\ldots,a^d)(\bf z^1,\ldots,\bf z^d)= Y(a^1,\bf z^1)\ldots Y(a^d,\bf z^d)\mathds 1.
    \end{equation}
\end{thm}
\begin{proof}
    Let $V$ be a CVA, and let $d\geq 1$ and $a^1,\ldots a^d\in V$.
    By mutual locality, for all $v\in V$ and $\sigma\in S_d$ the series
    \[\chi(\sigma,a^1,\ldots,a^d)Y(a^{\sigma(1)},\bf z^{\sigma(1)})\ldots Y(a^{\sigma(d)},\bf z^{\sigma(d)})v\]
    is the expansion of the same element $\mu_d(a^1,\ldots,a^d)(\bf z^1,\ldots,\bf z^d)$ in 
    $H^\bullet \Conf^\circ_d(V)$, with the order of all poles on diagonals uniformly bounded for all $v\in V$. This shows property (3).
    Axioms (1), and (2) above follow from axioms V2, and V3 for CVAs respectively. Axiom (4) follows from an inductive application of the associativity property for CVAs.

    Given $\mu$ as satisfying axioms (1)-(4) above, we can define $Y$ and recover all of its properties as follows:
    \begin{enumerate}
        \item Recall that $\mu_2(a,b)(\bf z,\bf w)\in H^\bullet \Conf_2(V)$, which is generated by power series over $z,w$ and symbols $\Omega^{\bf m}_{\bf z-\bf w}$. Applying the expansion $i_{\bf z,\bf w}$ we map $\Omega^{\bf m}_{\bf z-\bf w}$ to $\partial_{\bf w}^{(\bf m)}\Delta_-(\bf z,\bf w)$. Letting $\bf w=\bf 0$ we have $\partial_{\bf w}^{(\bf m)}\Delta_-(\bf z,\bf w)|_{\bf w=\bf 0} = \Omega^{\bf m}_{\bf z}$. In abuse of notation, we write
        \[\mu_2(a,b)(\bf z,\bf 0) \eqdef i_{\bf z,\bf w}\mu_2(a,b)(\bf z,\bf w)|_{\bf w=\bf 0},\]
        so for all $a,b\in V$ we have
        \[\mu_2(a,b)(\bf z,\bf 0)\in \cal K^{\bf z}(V).\]
        So we define $Y(a,\bf z)b = \mu_2(a,b)(\bf z,\bf 0)$, and this shows V1.
        \item Using axiom (4) we have
        \[\mu_2(a,\mathds 1)(\bf z,\bf w) = \mu_1(a)(\bf z) = e^{\bf z\cdot \bf T}a.\]
        Setting $\bf w=\bf 0$, we obtain $Y(a,\bf z)\mathds 1 = e^{\bf z\cdot \bf T}a$, which shows V3.
        \item By the same reasoning as in the previous point and axiom (3) we have 
        \[\mu_2(\mathds 1,a)(\bf w,\bf z) = e^{\bf z\cdot \bf T}a,\]
        so setting $\bf z=\bf 0$ gives us $\mu_2(\mathds 1,a)(\bf z,\bf 0) = Y(\mathds 1,\bf z)a = a$. This shows V2.
        \item Let $a,b\in V$ be homogeneous. Axiom (3) tells us that
        \[\mu_3(a,b,c)(\bf z,\bf w,\bf x) = p(a,b)\mu_3(b,a,c)(\bf w,\bf z,\bf x),\]
        and axiom (4) implies
        % \[\mu_3(a,b,c)(\bf z,\bf w+\bf y,\bf x+\bf y) = \mu_2(a,\mu_2(b,c)(\bf w,\bf x))(\bf z,\bf y)\]
        \[\mu_3(a,b,c)(\bf z,\bf w,\bf 0) = \mu_2(a,\mu_2(b,c)(\bf w,\bf 0))(\bf z,\bf 0)\]
        so $Y(a,\bf z)Y(b,\bf w)c$ and $p(a,b)Y(b,\bf w)Y(a,\bf z)c$ are expansions of the same element in $H^\bullet\Conf_d(V)$. By Proposition \ref{prop:localanalytic} this implies that $[Y(a,\bf z),Y(b,\bf w)]c$ is local in the variables $(\bf z,\bf w)$ for all $a,b,c\in V$. Thus, the series $[Y(a,\bf z),Y(b,\bf w)]$ is local, showing V5.
        \item We have
        \[e^{\bf w\cdot \bf T}Y(a,\bf z)b = \mu_1(\mu_2(a,b)(\bf z,\bf 0))(\bf w) = i_{\bf z,\bf w}\mu_2(a,b)(\bf z+\bf w,\bf w).\]
        If we set $b=\mathds 1$, then the resulting expression is in $V[\![\bf z,\bf w]\!]$, so we can set $\bf z=-\bf w$. This implies that
        \[e^{\bf w\cdot \bf T}Y(a,-\bf w)\mathds 1 = e^{\bf w\cdot \bf T}e^{-\bf w\cdot \bf T}a = \mu_2(a,\mathds 1)(\bf w-\bf w,\bf 0) = \mu_1(a)(\bf 0) = a.\]
        Thus $e^{\bf w\cdot \bf T}e^{-\bf w\cdot \bf T} = \id_V$.
        \item 
        We also have
        \[Y(a,\bf z+\bf w)e^{\bf w\cdot \bf T}b = i_{\bf z,\bf w}\mu_2(a,b)(\bf z+\bf w,\bf w),\]
        which shows that $e^{\bf w\cdot \bf T}Y(a,\bf z)b = Y(a,\bf z+\bf w)e^{\bf w\cdot \bf T}b$ for all $a,b\in V$. By the previous point we then have $e^{\bf w\cdot \bf T}Y(a,\bf z)e^{-\bf w\cdot \bf T} = Y(a,\bf z+\bf w)$, showing V4.
    \end{enumerate}
\end{proof}
It is worth noting that although we left a much more general set of axioms, we only needed to use facts about $\mu_d$ for $d=0,1,2,3$ in the proof.
\begin{remark}
    We expect that a $(n,n-1)$-CVA presented in the above axioms may be used to define a translation-equivariant Beilinson-Drinfeld factorization algebra on $\A^n$ (which by \cite{FrancisGaitsgory} is equivalent to a translation-equivariant chiral algebra). A derived version of the above axioms might give an $\infty$-category that is equivalent to translation-equivariant factorization algebras on $\A^n$.
\end{remark}

\section{Examples of cohomological vertex algebras}\label{Sec5}
With the tools we developed for CVAs, we can now discuss the construction of examples using the identities from our version of the Dong-Li lemma and various $k$-product identities. The main result in this section is the reconstruction theorem (\ref{thm:reconstruction}), which we use to provide explicit constructions of some examples. In particular, we give analogs of the Virasoro, the $\beta\gamma$-free field (or bosonic ghost), the Heisenberg, and the affine Kac-Moody VAs. We conclude by describing the BRST reduction of a CVA.

\subsection{The Reconstruction Theorem}
In this section, we prove an analog of the famous reconstruction theorem for VAs as well as provide some basic examples of CVAs. The reconstruction theorem was discovered nearly simultaneously by different people and is presented differently depending on the author (for example \cite{frenkel2004vertex}, \cite{kac1998vertex}, \cite{LepowskyLi} all present the reconstruction theorem slightly differently).
\begin{thm}[Reconstruction]\label{thm:reconstruction}
    Let $R$ be a ring. Suppose we are given the following data:
    \begin{itemize}
        \item A $\Z$-graded $R$-module $V=\bigoplus_{r\in \Z}V^r$.
        \item A nonzero element $\mathds 1\in V^0$.
        \item A collection of degree 0 endomorphisms $(T^{(\bf k)})_{\bf k\geq \bf 0}\subset \End(V)$.
        \item An $R$-submodule $\Sigma\leq V$ together with a degree 0 $R$-linear assignment
        \[Y_0 \colon \Sigma \to \cal E_c(V), \qquad a\mapsto Y_0(a,\bf z) \eqdef a(\bf z)\]
        in the sense that for $a=\sum_ra^r$ decomposed into homogeneous parts the corresponding field is $a(\bf z)=\sum_ra_r(\bf z)$ such that $a_r(\bf z)$ is homogeneous and $p(a_r) = p(a_r(\bf z)) = r$ for all $r\in \Z$.
    \end{itemize}
    Suppose the following holds:
    \begin{enumerate}
        \item For all $a\in V$ we have $e^{\bf w\cdot \bf T}a(\bf z)e^{-\bf w\cdot \bf T} = a(\bf z+\bf w)$.
        \item For all $a\in V$ we have $a(\bf z)\mathds 1 = e^{\bf z\cdot \bf T}a$.
        \item For all $a,b\in V$ the series $a(\bf z)$ and $b(\bf w)$ are mutually local.
        \item $V$ is generated by the symbols $a^{1}_{(-\bf 1-\bf k^1)}\ldots a^{d}_{(-\bf 1-\bf k^d)}\mathds 1$ with $\bf k^s\geq \bf 0$ and $a^i\in \Sigma$.
    \end{enumerate}
    Then the definition of $Y$ on the generating set of $V$ given by
    \begin{equation}
    \begin{aligned}
        Y(a_{(-\bf 1-\bf k^1)}^{1}\ldots a_{(-\bf 1-\bf k^d)}^{d}\mathds 1,\bf z) &\eqdef a^1(\bf z)_{(-\bf 1-\bf k^1)}\left(a^2(\bf z)_{(-\bf 1-\bf k^2)}\left(\ldots a^d(\bf z)_{(-\bf 1-\bf k^d)}\id_V\right)\right)\\
        &=\ \normord{\partial^{(\bf k^1)}_{\bf z}a^1(\bf z)\ldots \partial^{(\bf k^d)}_{\bf z}a^d(\bf z)}
    \end{aligned}
    \end{equation}
    gives a well-defined structure of a cohomological vertex algebra on $V$ such that the vacuum is $\mathds 1$, the translation operator is $e^{\bf z\cdot \bf T}$, and $a(\bf z) = Y_0(a,\bf z) = Y(a,\bf z)$ for all $a\in \Sigma$. Moreover, this structure is the unique extension of $Y_0$ with respect to the generating set of $V$ given.
\end{thm}
\begin{proof}
    We verify each axiom.
    \begin{enumerate}[label = (\roman*)]
        \item The fact that $Y$ maps every element of $V$ to a field is guaranteed by the Dong-Li lemma (\ref{lem:Dongpoly}), showing V1.
        \item The fact that $Y(\mathds 1,\bf z) = \id_V$ follows from the definition of $Y$, showing V2.
        \item Applying the Dong-Li lemma inductively on the generating set of $V$ shows V5.
        \item Using properties (1), (2), and V2 we have that $e^{\bf z\cdot \bf T}e^{\bf w\cdot \bf T}a=e^{(\bf z+\bf w)\cdot \bf T}a$ for any $a\in \Sigma$. Using the definition of the $k$-products and that $e^{\bf z\cdot \bf T}e^{-\bf z\cdot \bf T}a = a$ we can see that
        \begin{align*}
            e^{\bf w\cdot \bf T}\left(a^1(\bf z)_{(\bf k)}a^2(\bf z)\right)e^{-\bf w\cdot \bf T} 
            % \\&= \left(e^{\bf w\cdot \bf T}a^1(z)e^{-\bf w\cdot \bf T}\right)_{(k)}\left(e^{\bf w\cdot \bf T}a^2(z)e^{-\bf w\cdot \bf T}\right)\\
            &= a^1(\bf z+\bf w)_{(\bf k)}a^2(\bf z+\bf w).
        \end{align*}
        By induction, axiom V4 holds for any $v$ in the generating set of $V$.
        \item First note that for $\bf k\geq \bf 0$ and $a^1,a^2\in \Sigma$
        % \begin{align*}
        %     \left(a^1(w)_{(k)}a^2(w)\right)\mathds 1 &= \Res_z\left((\bf z-\bf w)^ka^1(z)e^{\bf w\cdot \bf T}a^2\right)\\
        %     &= e^{\bf w\cdot \bf T}\Res_z\left((\bf z-\bf w)^ka^1(\bf z-\bf w)a^2\right)\\
        %     &= e^{\bf w\cdot \bf T}a^1_{(k)}a^2.
        % \end{align*}
        % Similarly we can verify that
        \[\left(a^1(\bf w)_{(-\bf 1-\bf k)}a^2(\bf w)\right)\mathds 1 = e^{\bf w\cdot \bf T}a^1_{(-\bf 1-\bf k)}a^2.\]
        By induction, we can see that axiom V2, $Y(v,\bf z)\mathds 1 = e^{\bf z\cdot \bf T}v$ holds for any element $v$ of the generating set.
    \end{enumerate}
    This shows that $Y$ is well-defined on every $v$ in the generating set of $V$. Now for an arbitrary element $v\in V$ we write it as a finite linear combination $v = \sum_ir_iv_i$ of elements $v_i$ of the generating set and define $Y(v,\bf z) = \sum_ir_iY(v_i,\bf z)$. The only remaining concern here is that this formula could depend on the chosen presentation of $v$.
    
    Suppose we have two different formulas for $a\in V$ in terms of the generating set of $V$, say $a$ can be written as two different linear combinations $A^1$ and $A^2$ of elements of the generating set. We then wish to show that $Y(A^1,\bf z)=Y(A^2,\bf z)$ if and only if $A^1=A^2$, showing that $Y$ is well-defined. By linearity we can bring both formulas into the same argument, meaning that we wish to show that $Y(v,\bf z) = 0$ if and only if $v=0$. If $Y(v,\bf z)=0$, then by V2 we have $e^{\bf z\cdot \bf T}v = 0$, implying $v=0$. Conversely we note that if $v=0$ then $Y(v,\bf z)\mathds 1 = e^{\bf z\cdot \bf T}v = 0$. The proof of Goddard uniqueness may then be used here, allowing us to conclude that $Y(v,\bf z)\mathds 1 = 0$ implies $Y(v,\bf z)=0$. Therefore $Y(v,\bf z)$ is independent of the presentation of $v$ in terms of the generating set.
\end{proof}
\begin{defn}
    Let $V$ be a cohomological vertex algebra. A collection of (homogeneous) elements $\{a^\alpha\}_{\alpha\in I}\subset V$ is said to be \textit{(homogeneous) generators} of $V$ when $V$ is generated by elements of the form $a^{\alpha^1}_{(\bf k^1)}\ldots a^{\alpha^d}_{(\bf k^d)}\mathds 1$ for $d\geq 0$, $\alpha^i\in I$, and $\bf k^i\leq -\bf 1$ or $\bf k^i\geq \bf 0$. When this holds when $\bf k^i\leq -\bf 1$ for all $i$ (that is omitting $\bf k^i\geq \bf 0$), then the collection is said to be \textit{strong generators} of $V$.
\end{defn}
All of the examples of cohomological vertex algebras we provide will be strongly generated. Note that the endomorphisms $T^{(\bf k)}$ are determined by property (2), so property (1) is the only requirement imposed on the series $e^{\bf z\cdot \bf T}$.

\subsubsection{Constructing cohomological vertex algebras from mutually local fields}
Fix integers $n,N$ as before. Given a collection of abstract mutually local homogeneous cohomological series $a(\bf z)$ for all $a\in I$ for some index set $I$, we may define the module $V$ on which the fields $a(\bf z)$ act to be generated by strings of the form $s = a^1_{(-\bf 1-\bf k^1)}\ldots a^d_{(-\bf 1-\bf k^d)}\mathds 1$ for some $d\geq 0$, $a^i\in I$ and $\bf k^i\geq \bf 0$ for all $i$. From this definition, given $s$ we have $a_{(\bf k)}s = 0$ whenever any component of $k$ is sufficiently large, showing $\{a(\bf z)\}_{a\in I}$ form a set of mutually local fields on $V$. These can be used to give a CVA structure on $V$ by defining $T^{(\bf k)}a \eqdef a_{(-\bf 1-\bf k)}\mathds 1$ for $\bf k\geq \bf 0$. 

In the above context, we say that the CVA $V$ is \textit{freely generated} by the collection $\{a(\bf z)\}_{a\in I}$. All of the following examples are freely generated CVAs.

\subsection{The cohomological Virasoro vertex algebra}
The first noncommutative example of a cohomological vertex algebra we describe is the \textit{cohomological Virasoro vertex algebra} $V_{CW_n}$ over $R[\mathbf k]$ for a ring $R$ (ignoring $\mathbf k$ when $n\in \Z_{\geq 2}$). In practice, we should set $R$ to be a $\Q$-algebra such as $\C$.

We define the \textit{cohomological Witt algebra} $CW_n$ to be the Lie algebra of vector fields
\begin{equation}
    CW_n = \bigoplus_{a=1}^n\cal K^{\bf x}_{\mathrm{poly}}(R)\partial_{x_a}.
\end{equation}
There are other possible derivations of $\cal K^{\bf x}_{\mathrm{poly}}$ that exist due to the zero-divisors in $\cal K^{\bf x}_{\mathrm{poly}}$, but $CW_n$ consists of the naturally defined vector fields from an analytic perspective.

This $R$-module has a basis given by elements $G_{\bf r}^a \eqdef -\bf x^{\bf r}\partial_{x_a}$ and $\Gamma_{\bf r}^a \eqdef -\Omega_{\bf x}^{\bf r}\partial_{x_a}$ for $\bf r\geq \bf 0$ and $a\in [1,n]$, where $G_{\bf k}^a$ has degree 0 and $\Gamma_{\bf r}^a$ has degree $N$. These satisfy the bracket relations
\[[G^a_{\bf r},G^b_{\bf s}] = r_bG^a_{\bf r+\bf s-\bf e_b} - s_aG^b_{\bf r+\bf s-\bf e_a},\qquad [\Gamma^a_{\bf r},\Gamma^b_{\bf s}] = 0,\]
\[[G^a_{\bf r}, \Gamma^b_{\bf s}] = r_b\Gamma_{\bf s-\bf r+\bf e_b}^a + (s_a+1)\Gamma_{\bf s-\bf r+\bf e_a}^b.\]
As usual, we set $G_{\bf r}^a = \Gamma_{\bf r}^a = 0$ if $\bf r\not\geq \bf 0$.
These bracket formulas provide $CW_n$ with a natural Lie algebra structure over $R$, and although we can naturally assign a grading to the Lie algebra it is not present in any of the bracket relations. 
We may also arrange these modes into a collection of series $(\Gamma^a(\bf z))_a$ via the definition
\begin{equation}
    \Gamma^a(\bf z) = \sum_{\bf k\geq \bf 0}\bf z^{\bf k}(-1)^N\Gamma_{\bf k}^a + \Omega^{\bf k}_{\bf z} G_{\bf k}^a = -\Delta(\bf z,\bf x)\partial_{x_a}.
\end{equation}
This is similar to the classical case where we set $T(z) = \sum_{k}z^{-k-2}L_k = -\sum_kz^{-1-k}x^k\partial_x = -\delta(z,x)\partial_x$. We can rescale the elements $\Gamma_{\bf k}^a$ together by a unit, so the extra $(-1)^N$ factor may be absorbed into $\Gamma_{\bf k}^a$ above.
The bracket relations for $CW_n$ are equivalent to the following graded commutator formula:
\begin{equation}\label{eq:CWncommutatorformula}
    [\Gamma^a(\bf z),\Gamma^b(\bf w)] = 
 \partial_{w_a}\Delta(\bf z,\bf w)\Gamma^b(\bf w) + 
 \partial_{w_b}\Delta(\bf z,\bf w)\Gamma^a(\bf w)+ \Delta(\bf z,\bf w)\partial_{w_a}\Gamma^b(\bf w).
\end{equation}
The 1-dimensional central extensions of $CW_n$ are quite restrictive:
\begin{lem}\label{thm:CWnCentral}
    Let $R$ be a $\Q$-algebra. The Lie algebra $CW_1$ admits a one-dimensional universal central extension $\widetilde{CW}_1 = CW_1\oplus R\mathbf c$, where the bracket is given as
    \begin{equation}
        [G_r,\Gamma_s] = (r+s+1)\Gamma_{s-r+1} + \delta_{r,s+3}\binom{s+3}{3}\mathbf c.
    \end{equation}
    $CW_n = \widetilde{CW}_n$ is centrally closed (as a Lie algebra) for $n\geq 2$.    
\end{lem}
One may examine a general 2-cocycle and use recursive equations to show the result. Alternatively, one may restrict to $R=\C$ or $\R$ and use Gelfand-Fuks cohomology for formal vector fields \cite{GelfandFuksFormal}.
For $n=1$, the extended bracket for $\Gamma(z)$ is
\begin{equation}\label{eq:CW1relations}
    [\Gamma(z),\Gamma(w)] = \Delta(z,w)\partial_{w}\Gamma(w) + 2\partial_{w}\Delta(z,w)\Gamma(w) + \partial^{(3)}_w\Delta(z,w)\mathbf c.
\end{equation}

For convenience, we set the underlying ring to be $R=\C$. For $n\in \Z_{\geq 2}$, the CVA $V_{\widetilde{CW}_n}$ is strongly generated over $\C$ by the fields $\{\Gamma^a(\bf z)\}_{a=1}^n$ satisfying the commutator relations in equation \ref{eq:CWncommutatorformula}. Similarly, for $n=1$ the CVA $V_{\widetilde{CW}_1}$ is strongly generated over $\C[\mathbf c]$ by the field $\Gamma(z)$ with relations given by equation \ref{eq:CW1relations}. Explicitly, the underlying vector space $V_{\widetilde{CW}_n}$ is given by
\[V_{\widetilde{CW}_n} = \Span_R\left\{\Gamma_{\bf k^1}^{a^1}\ldots\Gamma_{\bf k^d}^{a^d}\mathds 1\right\}_{\bf k^i\in \N^n,\
a\in [1,n],\
d\in \Z_{\geq 1}},\]
where $R=\C[\mathbf c]$ for $n=1$ and $\C$ otherwise; and the fields are generated by the $\bf k$-product relations given in the reconstruction theorem together with $\Gamma^a(\bf z) = Y(\omega^a,\bf z)$ for $\omega^a\eqdef \Gamma_{\bf 0}^a\mathds 1$. Here $\omega^{a}_{(\bf k)} = G^a_{\bf k}$ and $\omega^a_{(-\bf 1-\bf k)} = \Gamma_{\bf k}^a$. In this description, $\C[\mathbf c]$ should be treated as a graded commutative algebra, with $\mathbf c$ having degree $N$. If we quotient by the maximal ideal $(\mathbf c)$, then we naturally obtain the quotient CVA $V_{\widetilde {CW}_1}\twoheadrightarrow V_{CW_n}$.
\begin{remark}
    One issue with the above construction for $n=1$ is that the degree $N$ central element $\mathbf c$ is not invertible, which is an important feature for VOAs. If one wishes, they can resolve this issue by replacing $R$ with a graded commutative algebra where $\mathbf c$ is invertible, such as $R = \C[\mathbf c,\mathbf c^{-1}]$ or $\C(\mathbf c)$. The latter may be preferable since it is a field, and in either case $\mathbf c^{-1}$ has degree $-N$. For this we would have to require $N$ to be even since $1 = \mathbf c\mathbf c^{-1} = (-1)^N\mathbf c^{-1}\mathbf c = (-1)^N$. 
\end{remark}
We will explicitly describe the actions of the modes $\omega^a_{(\bf e_b)}$ and $\omega^a_{(\bf 0)}$ on $V_{CW_n}$. Using the commutator formula and that $Y(T^{(\bf k)}a,\bf z) = \partial_{\bf z}^{(\bf k)}Y(a,\bf z)$, we have
\begin{equation}
\begin{aligned}
    \omega_{(\bf j)}^a\left(\Gamma_{\bf k^1}^{b^1}\ldots\Gamma_{\bf k^d}^{b^d}\mathds 1\right)
    % &= \omega_{(j)}^a\left((T^{(k^1)}\omega^{b^1})_{(-1)}\ldots(T^{(k^d)}\omega^{b^d})_{(-1)}\mathds 1\right)\\
    =\sum_{i=1}^d(T^{(\bf k^1)}\omega^{b^1})_{(-\bf 1)}\ldots (\omega_{(\bf j)}^aT^{(\bf k^i)}\omega^{b^i})_{(-\bf 1)}\ldots (T^{(\bf k^d)}\omega^{b^d})_{(-\bf 1)}\mathds 1.
\end{aligned}
\end{equation}
Here we also note that $\omega^a_{(\bf j)}\mathds 1 = 0$. Since $Y(T^{(\bf k)}\omega^b,\bf w) = \partial_{\bf w}^{(\bf k)} Y(\omega^b,\bf w)$ we have the commutator formula
\begin{equation*}
\begin{aligned}
    &[Y(\omega^a,\bf z),Y(T^{(\bf k)}\omega^b,\bf w)] \\
    % &= (k_a+1)\partial_w^{(k+e_a)}\left[\Delta(\bf z,\bf w)Y(\omega^b,w)\right] + 
    % \partial_w^{(k)}\left[\partial_{w_b}\Delta(\bf z,\bf w)Y(\omega^a,w)\right]\\
    &= (k_a+1)\sum_{\bf j\geq \bf 0}\partial^{(\bf j)}_{\bf w}\Delta(\bf z,\bf w)\partial^{(\bf k+\bf e_a-\bf j)}_{\bf w}Y(\omega^b,\bf w) + \sum_{\bf j\geq \bf 0}\partial^{(\bf j)}_{\bf w}\partial_{w_b}\Delta(\bf z,\bf w) \partial^{(\bf k-\bf j)}_{\bf w}Y(\omega^a,\bf w)\\
    &= \Delta(\bf z,\bf w) \partial_{w_a}\partial^{(\bf k)}_{\bf w}Y(\omega^b,\bf w) 
    + \partial_{w_b}\Delta(\bf z,\bf w)\partial_{\bf w}^{(\bf k)}Y(\omega^b,\bf w)\\
    &\quad + (k_a+1)\sum_{i=1}^d\partial_{w_i}\Delta(\bf z,\bf w)\partial_{\bf w}^{(\bf k+\bf e_a-\bf e_i)}Y(\omega^a,\bf w) + \ldots
\end{aligned}
\end{equation*}
Reading off the commutator expansion and denoting $T_a = T^{(\bf e_a)}$ as usual, we have
\begin{align}
    \omega^{a}_{(\bf 0)}\left(T^{(\bf k)}\omega^b\right) &= T_a\left(T^{(\bf k)}\omega^b\right),\\
    \label{eq:commutatorsCWn}\omega^{a}_{(\bf e_c)}\left(T^{(\bf k)}\omega^b\right) &= \delta_{b,c}\left(T^{(\bf k)}\omega^a\right) + (k_a+1)T^{(\bf k+\bf e_a-\bf e_c)}\omega^b.
\end{align}
It is then straightforward to check by induction via the commutator formula and the $k$-product formula that
\begin{align}
    \omega^a_{(\bf 0)}\left(\Gamma_{\bf k^1}^{b^1}\ldots\Gamma_{\bf k^d}^{b^d}\mathds 1\right) &= T_a\left(\Gamma_{\bf k^1}^{b^1}\ldots\Gamma_{\bf k^d}^{b^d}\mathds 1\right),\\
    \omega^a_{(\bf e_a)}\left(\Gamma_{\bf k^1}^{b^1}\ldots\Gamma_{\bf k^d}^{b^d}\mathds 1\right) &= \sum_{i=1}^d(1+\delta_{a,b^i}+k^i_a)\Gamma_{\bf k^1}^{b^1}\ldots\Gamma_{\bf k^d}^{b^d}\mathds 1.
\end{align}
Thus $\omega^a_{(\bf 0)} = T_a$ for all $a$ and the collection of operators $(\omega^a_{(\bf e_a)})_a$ are simultaneously diagonalizable. This gives us an $\N^n$-spin grading $\deg$ on $V_{CW_n}$ via the assignment
\begin{equation}
    \deg(\Gamma^{a^1}_{\bf k^1}\ldots \Gamma^{a^d}_{\bf k^d}\mathds 1) = \sum_{i=1}^d \bf 1 + \bf e_{a^i} + \bf k^i.
\end{equation}
We write this spin grading as $V_{CW_n} = \bigoplus_{\bf s\geq \bf 0}V_{CW_n}^{(\bf s)}$.
As before we denote $\deg_i(v)$ for the $i$th component of $\deg(v)$.
In terms of commutator formulas, we can then write for every spin homogeneous $v\in V$ of $\omega^i_{(e_i)}$
\begin{align}
    [\omega^{i}_{(\bf e_i)},Y(v,\bf z)] = \left(z_i\partial_{z_i} + \deg_i(v)\right)Y(v,\bf z).
\end{align}
In particular, $\deg \omega^a = \bf 1+\bf e_{a}$.

The following proposition describes an important feature of $V_{CW_n}$ for $n>1$.
\begin{prop}
    For every $v\in V$ there is a finite-dimensional subspace $V_v\ni v$ such that $V_v$ is invariant under the action of $\omega^i_{(\bf e_j)}$ for all $i$ and $j$, that is $\omega^i_{(\bf e_j)}V_v\subset V_v$.
\end{prop}
\begin{proof}
     By \eqref{eq:commutatorsCWn}, for $v = \Gamma_{\bf k}^a\mathds 1 = T^{(\bf k)}\omega^a$ we can set $V_{|\bf k|}$ to be the space of all elements of the form $T^{(\boldsymbol\ell)}\omega^b$ for $b\in [1,n]$ and $\boldsymbol\ell\geq \bf 0$ such that $|\boldsymbol\ell|=|\bf k|$. This is finite-dimensional (the dimension is $n\binom{n+|\bf k|-1}{n-1}$) and the action of $\omega^{i}_{(\bf e_j)}$ maps $V_{|\bf k|}$ into itself.
     More generally, for a string of the form $\Gamma_{\bf k^1}^{a^1}\ldots \Gamma_{\bf k^d}^{a^d}\mathds 1$ we can set the invariant subspace $V_{|\bf k^1|,\ldots,|\bf k^d|}$ to be the set of strings of the form $s = \Gamma_{\boldsymbol\ell^1}^{b^1}\ldots \Gamma_{\boldsymbol\ell^d}^{b^d}\mathds 1$ in which $|\boldsymbol\ell^i|=|\bf k^i|$ for all $i$ and $b^i\in [1,n]$. A larger subspace would be to take the set of strings such that $|\boldsymbol\ell^i| = \max_j|\bf k^j|$ for all $i$. Such larger subspaces also contain linear combinations of strings $s_1 + \ldots + s_m$.
\end{proof}
When $n=1$ then we only have the mode $\omega_{(1)}$, which is diagonalizable. For an element $v\in V_{CW_n}$ we can take the eigenvector decomposition of $v$, then form $V_v$ to be the subspace spanned by those eigenvectors, which is finite-dimensional. Therefore, this proposition is nontrivial only for $n>1$.

\subsubsection{Cohomological vertex operator algebras}\label{subsubsec:conformal}
Since we have described the CVA $V_{\widetilde{CW}_n}$, we remark that it has multiple properties similar to the definition of a VOA. 
This leads us to give the following definitions: 
\begin{defn}
    Let $R$ be a field of characteristic 0. We say a CVA $V$ over $R[\mathbf c]$ for $n=1$ and $R$ for $n>1$ with a spin grading is \textit{conformal} if it has a collection of homogeneous vectors $\{\omega^i\}_{i=1}^n$, called \textit{conformal vectors}, such that the following hold:
    \begin{enumerate}
        \item The fields $\Gamma^i(\bf z)$ satisfy the relations of $\widetilde{CW}_n$ with central charge $\mathbf c$ of cohomological degree $N$.
        \item $\omega^i_{(\bf 0)} = T_i$,
        \item $\omega^i_{(\bf e_i)} = \deg_i$.
    \end{enumerate}

    Furthermore, we say $V$ is a \textit{cohomological vertex operator algebra} (CVOA) if $V$ is conformal and the following extra conditions hold:
    \begin{enumerate}
        
        \item The components $V^{r,(\bf s)}$ of $V$ are allfinite-dimensional.
        \item Given $a\in V$, there is a finite-dimensional subspace $V_a\ni a$ such that $\omega^i_{(\bf e_j)}V_a\subset V_a$ for all $i,j$.
    \end{enumerate}
    A CVOA $V$ is of \textit{CFT-type} if the spin grading is an $\N^n$-grading and $V^{0,(\bf 0)} = \C\mathds 1$.
\end{defn}
We note that the central charge $\mathbf c$ is usually replaced with $c/2$ in the literature on the Virasoro algebra. Similar to 2D CFT, this central charge $\mathbf c$ measures the conformal anomaly; unlike VAs, the central charge $\mathbf c$ has cohomological degree $N$, so its behavior is quite different. In fact, the central charge $c$ for a VOA naturally appears when applying normal ordered products due to the non-associative, non-commutative nature of the normal ordered product for VAs. For CVAs, the normal ordered product is associative and commutative, and no extra central charge term appears in most formulas.

\subsection{An analog of the $\beta\gamma$-system}
One of the simplest cases of locality is as follows:
Let $\beta$ and $\gamma$ be mutually local homogeneous series such that for some $c\in R$ we have
\[[\gamma(\bf z),\beta(\bf w)] = \Delta(\bf z,\bf w)c,\qquad [\beta(\bf z),\beta(\bf w)] = [\gamma(\bf z),\gamma(\bf w)] = 0.\]
The right side having degree $N$ implies that $p(\gamma) + p(\beta) = N$. 
We define $V_{\beta\gamma}^c$ over $R$ to be the CVA freely generated by $\beta$ and $\gamma$, which we call the \textit{$\beta\gamma$-CVA} (after a choice of degree $p(\beta)$, the number of variables $n$, and the degree shifting $N$). For simplicity we set $c=1$, in which case we write $V_{\beta\gamma} \eqdef V^1_{\beta\gamma}$.

We define the series
\begin{equation*}
    B^i(\bf z) =\, \normord{\gamma(\bf z)\partial_{z_i}\beta(\bf z)},\qquad G^i(\bf z) =\, \normord{\beta(\bf z)\partial_{z_i}\gamma(\bf z)}.
\end{equation*}
Repeatedly applying the Dong-Li lemma gives us
\begin{align*}
    [B^i(\bf z), B^j(\bf w)]
    &=(-1)^{p(\beta)}\Big(\partial_{w_i}\Delta(\bf z,\bf w)B^j(\bf w) + \partial_{w_j}\Delta(\bf z,\bf w)B^i(\bf w)+ \Delta(\bf z,\bf w)\partial_{w_i}B^j(\bf w)\Big),
\end{align*}
Therefore $\Gamma^i(\bf z) \eqdef (-1)^{p(\beta)}B^i(\bf z)$ satisfies the commutator relations of $CW_n$, giving a conformal structure on $V_{\beta\gamma}$.

When $(-1)^{p(\beta)p(\gamma)} = 1$, many formulas simplify. Here are a few:
\begin{align*}
    [G^i(\bf z), G^j(\bf w)]
    &=(-1)^{1+p(\beta)}\Big(\partial_{w_i}\Delta(\bf z,\bf w)G^j(\bf w) + \partial_{w_j}\Delta(\bf z,\bf w)G^i(\bf w)+ \Delta(\bf z,\bf w)\partial_{w_i}G^j(\bf w)\Big),\\
    [B^i(\bf z),G^j(\bf w)] 
    % &= (-1)^{p(\beta)}\Big(\partial_{w_j}\partial_{w_i}\Delta(\bf z,\bf w)\normord{\beta(w)\gamma(w)} + \partial_{w_i}\Delta(\bf z,\bf w)G^j(w)\\ 
    % &\qquad\qquad\qquad + \partial_{w_j}\Delta(\bf z,\bf w)G^i(w) + \Delta(\bf z,\bf w)\partial_{w_i}G^j(w)\Big)\\
    &= (-1)^{p(\beta)}\partial_{w_j}\partial_{w_i}\Delta(\bf z,\bf w)\normord{\beta(\bf w)\gamma(\bf w)} - [G^i(\bf z),G^j(\bf w)],\\
    [G^i(\bf z),B^j(\bf w)]
    % &=(-1)^{1+p(\beta)}\Big(\partial_{w_j}\partial_{w_i}\Delta(\bf z,\bf w)\normord{\gamma(w)\beta(w)} + \partial_{w_i}\Delta(\bf z,\bf w)B^j(w) \\
    % &\qquad\qquad\qquad + \partial_{w_j}\Delta(\bf z,\bf w)B^i(w) + \Delta(\bf z,\bf w)\partial_{w_i}B^j(w)\Big)\\
    &=(-1)^{1+p(\beta)}\partial_{w_j}\partial_{w_i}\Delta(\bf z,\bf w)\normord{\gamma(\bf w)\beta(\bf w)} - [B^i(\bf z),B^j(\bf w)]\\
    [J(\bf z),J(\bf w)] &= 0.
\end{align*}
\begin{remark}
    Let $n=1$. Some of these commutator relations should appear unusual to those familiar with the usual $\beta\gamma$-VA.
    For the usual $\beta\gamma$-VA, the field $J(z) =\, \normord{\beta(z)\gamma(z)}$ has the commutator relation $[J(z),J(w)] = -\partial_w\delta(z,w)$, hence generates the Heisenberg VOA. That is not the case here, as the bracket is instead 0. Similarly $n=1$, the field $B(z) = \normord{\gamma(z)\partial_z\beta(z)}$ has a nonzero central charge for the $\beta\gamma$-VA, whereas it has a vanishing central charge for the $\beta\gamma$-CVA. This phenomenon is related to the normal ordered product being associative and graded commutative for CVAs, whereas it is neither associative nor commutative for VAs.
\end{remark}
From here on, we require that either $p(\beta)$ or $p(\gamma)$ be even so that $(-1)^{p(\beta)p(\gamma)} = 1$.
Given $\lambda\in R$, we define
\begin{equation}
    \Gamma_\lambda^i(\bf z) \eqdef (-1)^{p(\beta)}[(1-\lambda) B^i(\bf z) -\lambda G^i(\bf z)].
\end{equation}
Note that $\Gamma^i_\lambda(\bf z) = Y(\omega^{i,\lambda},z)$, where
\[\omega^{i,\lambda} = (1-\lambda)\gamma_{(-1)}\beta_{(-\bf 1-\bf e_i)}\mathds 1 - \lambda\beta_{(-\bf 1)}\gamma_{(-\bf 1-\bf e_i)}\mathds 1.\]
From all of the above identities, the series $\{\Gamma^{i}_\lambda(\bf z)\}_{i=1}^n$ again obey the $CW_n$ commutation relations. Bracketing with $\gamma$ and $\beta$, we get
\begin{align*}
    [\Gamma_\lambda^i(\bf z) , \gamma(\bf w)] 
    &= \left((1-\lambda)\partial_{w_i}\Delta(\bf z,\bf w)\gamma(\bf w) + \Delta(\bf z,\bf w)\partial_{w_i}\gamma(\bf w)\right)\\
    [\Gamma_\lambda^i(\bf z) , \beta(\bf w)] 
    &= \left(\lambda\partial_{w_i}\Delta(\bf z,\bf w)\beta(\bf w) + \Delta(\bf z,\bf w)\partial_{w_i}\beta(\bf w)\right).
\end{align*}

The above identities tell us that for $i\neq j$
\begin{align*}
    &\omega^{i,\lambda}_{(\bf 0)}\gamma_{(-\bf 1)}\mathds 1 = T^{(\bf e_i)}\gamma_{(-\bf 1)}\mathds 1
    &&\omega^{i,\lambda}_{(\bf e_i)}\gamma_{(-\bf 1)}\mathds 1 = (1-\lambda)\gamma_{(-\bf 1)}\mathds 1\\
    &\omega^{i,\lambda}_{(\bf 0)}\beta_{(-\bf 1)}\mathds 1 = T^{(\bf e_i)}\beta_{(-\bf 1)}\mathds 1
    &&\omega^{i,\lambda}_{(\bf e_i)}\beta_{(-\bf 1)}\mathds 1 = \lambda\beta_{(-\bf 1)}\mathds 1\\
    &\omega^{i,\lambda}_{(\bf 0)}\omega^{j,\lambda} = T^{(\bf e_i)}\omega^{j,\lambda}
    &&\omega^{i,\lambda}_{(\bf 0)}\omega^{i,\lambda} = T^{(\bf e_i)}\omega^{j,\lambda},\\
    &\omega^{i,\lambda}_{(\bf e_j)}\omega^{j,\lambda} = \omega^{i,\lambda}
    &&\omega^{i,\lambda}_{(\bf e_i)}\omega^{j,\lambda} = \omega^{i,\lambda}.
\end{align*}
By induction, we can see that $\omega^{i,\lambda}_{(\bf 0)} = T^{(\bf e_i)} = T_i$ for all $i$.
If the base ring is $R=\C$ and $\lambda = 0$ or 1, and $\{\omega^{i,\lambda}_{(\bf e_i)}\}_{i=1}^n$ induces an $\N^n$-grading on $V$ (more generally we can let $\lambda = 1/L$ to get a $(\tfrac{1}{L}\N)^n$-grading.

We may assign another grading to $V_{\beta\gamma}$, sometimes called the \textit{flavor grading}. Define
\[J(\bf z) \eqdef \normord{\beta(\bf z)\gamma(\bf z)}\]
We have the following relations obtained by the Dong-Li lemma:
\begin{align*}
    [J(\bf z),\beta(\bf w)] &= (-1)^{p(\beta)+p(\beta)p(\gamma)}\Delta(\bf z,\bf w)\beta(\bf w),\\
    [J(\bf z),\gamma(\bf w)] &= (-1)^{1+p(\beta)}\Delta(\bf z,\bf w)\gamma(\bf w),\\
    [J(\bf z),J(\bf w)] &= (-1)^{p(\beta)+p(\beta)p(\gamma)}\Delta(\bf z,\bf w)J(\bf w) + (-1)^{1+p(\beta)}\Delta(\bf z,\bf w)J(\bf w).
\end{align*}
Now assume that $p(\beta)$ is even, in which case the right-hand side of the third equation is 0.
We assign a grading on $V_{\beta\gamma}$ by declaring all modes of $\beta(\bf z)$ to have flavor charge $+1$ and all modes of $\gamma(\bf z)$ to have flavor charge $-1$. This grading respects normal ordered products, so the modes of $J(\bf z)$ have flavor charge 0.

\subsection{A Heisenberg analog}
Here we mention a generalization of the Heisenberg vertex algebra given in \cite{ravioli}. In particular, we define $V^c_H$ to be the CVA generated by the fields $\gamma(\bf z)$ and $\{\partial_{z_i}\beta(\bf z)\}_{i=1}^n$ from the $\beta\gamma$-CVA $V_{\beta\gamma}^c$. We have a natural inclusion $H^c\subset V_{\beta\gamma}^c$. These fields obey the bracket relations
\begin{equation}
    [\gamma(\bf z),b^i(\bf w)] = \partial_{w_i}\Delta(\bf z,\bf w)c,\qquad [\gamma(\bf z),\gamma(\bf w)]=[b^i(\bf z),b^j(\bf w)]=0.
\end{equation}
where $c\in R$. We also impose the additional relation that $\partial_{z_j}b^i(\bf z) = \partial_{z_i}b^j(\bf z)$ for all $i,j$.
We refer to $H^c$ as \textit{Heisenberg CVAs}. For convenience, we will again set $c=1$. When $N=n=1$, this corresponds to the Heisenberg raviolo vertex algebra $H$ from \cite{ravioli}.

Note that the above relation implies that $p(\gamma)+p(b^i) = N$ for all $i$, so $p(b) \eqdef p(b^i)$ is independent of $i$. Now, we define the series $\Gamma^i(\bf z) = (-1)^{p(b)}\normord{\gamma(\bf z)b^i(\bf z)}$ following the previous section. The series $\{\Gamma^i(\bf z)\}_{i=1}^n$ obey the commutation relations of $CW_n$, and we have by the Dong-Li lemma
\begin{align*}
    [B^i(\bf z),\gamma(\bf w)] &= \partial_{w_i}\Delta(\bf z,\bf w)\gamma(\bf w) + \Delta(\bf z,\bf w)\gamma(\bf w)\\
    [B^i(\bf z),b^j(\bf w)] &= \partial_{w_j}\Delta(\bf z,\bf w)b^i(\bf w) + \Delta(\bf z,\bf w)\partial_{w_j}b^i(\bf w).
\end{align*}
We may substitute $\partial_{w_j}b^i(\bf w) = \partial_{w_i}b^j(\bf w)$ into the second equation. The above relations reduce to $B^i_{(\bf 0)} = T_i$ for all $i$ and
\begin{equation*}
    B^i_{(\bf e_i)}\gamma_{(-\bf 1)}\mathds 1 = \gamma_{(-\bf 1)}\mathds 1,\qquad
    B^i_{(\bf e_j)}b^j_{(-\bf 1)}\mathds 1 = b^j_{(-\bf 1)}\mathds 1.
\end{equation*}
The other modes of $B^i$ act by 0 on $\gamma_{(-\bf 1)}\mathds 1$ and $b^j_{(-\bf 1)}\mathds 1$.

\subsection{Affine vertex algebras}\label{subsec:affine}
For this section, the underlying ring $R$ is a $\C$-algebra.
When $A$ is a commutative $\C$-algebra and $\frak g$ is a simple complex Lie algebra, it is a result due to Kassel \cite{Kassel} that the universal central extension of $\frak g(A)\eqdef \frak g\otimes_\C A$ is
\[\widehat{\frak g} = \frak g(A)\oplus \Omega_{A}/dA,\]
where $\Omega_A$ is the vector space of K\"ahler differentials over $A$.
This can be thought of as a generalization of the affine Lie algebra $\widehat{\frak g} = \frak g\otimes_\C \C[t,t^{-1}] \oplus \C\mathbf c$ since the extension $\Omega_{\C[t,t^{-1}]}/d\C[t,t^{-1}]$ is one-dimensional.
Similarly, we can ask that $\frak g$ is a graded Lie algebra as well, but we leave this case for future work.

For our purposes, we define
\[\Omega(\cal K^{\bf t}_{\mathrm{poly}}) \eqdef \bigoplus_{i=1}^n\cal K^{\bf t}_{\mathrm{poly}}dt_i,\]
subject to the relations
\begin{equation*}
    0 = d(\bf t^{\bf k}) = \sum_{i=1}^nk_i\bf t^{\bf k-\bf e_i}dt_i,\qquad 
    0 = d(\Omega^{\bf k}_{\bf t}) = -\sum_{i=1}^n(k_i+1)\Omega_{\bf t}^{\bf k+\bf e_i}dt_i
\end{equation*}
This second identity is a chain rule identity similar to that of $d(\bf t^{\bf k})$; it does not hold in $\Omega_{\cal K^{\bf t}_{\mathrm{poly}}}/\cal K^{\bf t}_{\mathrm{poly}}$, so we must impose this relation. We can also naturally induce a cohomological $\Z$-grading on $\Omega(\cal K^{\bf t}_{\mathrm{poly}})$ from $\cal K^{\bf t}_{\mathrm{poly}}$.
For $n=1$, we have $t^{k}dt = 0$ for all $k\in \N$ and $\Omega^k_t\,dt = 0$ unless $k=0$. $\Omega(\cal K^{\bf t}_{\mathrm{poly}})$ is then one-dimensional, recovering the degree 1 central extension as described in \cite{ravioli}.

Now let $\frak g$ be a complex Lie algebra with a symmetric, invariant, $\C$-bilinear form $h (\cdot,\cdot)$.
We define the Lie algebra
\[\widehat{\frak g} = \cal K^{\bf t}_{\mathrm{poly}}(\frak g) \oplus \Omega(\cal K^{\bf t}_{\mathrm{poly}})\]
with the bracket
\[[X\otimes \alpha, Y\otimes \beta] = [X,Y]\otimes \alpha\beta - h (X,Y) \alpha\,d\beta.\]
We let $\{\mu_a\}_{a}$ be an orthonormal basis with respect to $h$, and let $\{f^c_{ab}\}_{a,b,c}$ be the structure constants of $\frak g$ so that $[\mu_a,\mu_b] = f^c_{ab}\mu_c$ (here we are using Einstein summation).
From this, we define the fields
\[X(\bf z) = \sum_{\bf k\geq \bf 0}\bf z^{\bf k}(X\otimes \Omega^{\bf k}_{\bf t}) + \Omega_{\bf z}^{\bf k}(X\otimes \bf t^{\bf k}),\qquad K^\ell(\bf z) = \sum_{\bf k\geq \bf 0}\bf z^{\bf k} \Omega^{\bf k}_{\bf t}\, dt_\ell + \Omega_{\bf z}^{\bf k} \bf t^{\bf k}\, dt_\ell.\]
The series $X(\bf z)$ and $K^\ell(\bf z)$ are naturally homogeneous of degree $N$ for all $X\in \frak g$ and $\ell\in [1,n]$.
We can naturally see that
\begin{equation}
    [X(\bf z),Y(\bf w)] = \Delta(\bf z,\bf w)[X,Y](\bf w) + h (X,Y)\sum_{\ell=1}^n\partial_{w_\ell}\Delta(\bf z,\bf w)K^\ell(\bf w),
\end{equation}
with $K^\ell(\bf w)$ central for all $\ell$. When $n=1$, we have $K^1(w) = \Omega^0_t\,dt$, as expected.

Denote the graded symmetric algebra of $\Omega(\cal K^{\bf t}_{\mathrm{poly}})$ as $\C[\Omega]$. Now we may define the \textit{affine CVA} $V[\frak g,h]$ over $\C[\Omega]$ to be the vertex algebra strongly generated by the fields $\mu_a(\bf z)$ and $K^\ell(\bf z)_+$ (the power series part) for all $X\in \frak g$ and $\ell\in [1,n]$. This is because the forms $\bf t^{\bf k}dt_i$ and $\Omega^{\bf k}_{\bf t}dt_i$ are all central, so $\bf t^{\bf k}dt_i$ act as 0 on $V_{\frak g}$. The natural next step in the classical setting would be to quotient the central elements in the vertex algebra so that they act on the vacuum by scalars, but we cannot do this here since $K^\ell_{(-\bf 1-\bf k)}$ is of cohomological degree $N$ for all $\bf k\geq \bf 0$. Moreover, the na\"ive Segal-Sugawara construction does not make sense here because the normal ordered product

As remarked in \cite{ravioli}, it is unclear whether there is a natural conformal structure on $V_{\frak g}$, and we suspect there is not. This suggests that to obtain interesting conformal cohomological vertex algebras (particularly for $n>1$), one must construct them in other ways, such as by looking at other models in perturbative gauge theory.

\subsection{BRST reduction and W-algebras}
For this section, we describe a slight generalization of the BRST reduction from \cite[Section 4.5.2]{ravioli}, which leads naturally to the definition of W-algebras. For simplicity, we will set the base ring to be $R=\C$ again. We start with some definitions.
\begin{defn}
    A \textit{derivation} $D$ of (cohomological) degree $r$ on a CVA $V$ is a homogeneous degree $r$ linear map $D \colon V\to V$ such that
    $[D,Y(a,\bf z)] = Y(Da,\bf z)$.
    $D$ is said to be a \textit{differential} if $D^2=0$ and $r=1$. A \textit{differential-graded (dg) CVA} $(V,d)$ is a CVA $V$ with a differential $d$.
\end{defn}
\begin{defn}
    A \textit{cohomological vertex superalgebra (CVSA)} is a CVA with an additional $\Z/2$-grading $V^r = V^r_{+}\oplus V^r_{-}$ for all $r\in \Z$, where $\mathds 1$, $(T^{(\bf k)})_{k\geq 0}$, and $Y$ are all even with respect to this decomposition. Here $V^r_+$ is the even part and $V^r_-$ is the odd part.
    In all formulas involving the Koszul rule of signs, we replace $p(a)$ with $P(a) = p(a) + \delta_\pm$ for $a\in V^r_\pm$, where $\delta_+ = 0$ and $\delta_- = 1$. A homogeneous element $a\in V$ is \textit{bosonic} (resp. \textit{fermionic}) if $P(a)$ is even (resp. odd) (similarly so for homogeneous operators).
\end{defn}
Let $V$ be a CVSA, and let $W\in V$ be a nonzero bosonic element of cohomological degree $N+1$ with $N$ odd. Then $W_{(\bf 0)}$ is a degree 1 fermionic operator on $V$. Then $W_{(\bf 0)}^2 = \frac{1}{2}[W_{(\bf 0)},W_{(\bf 0)}] = 0$ if and only if $(W_{(\bf 0)}W)_{(\bf 0)} = 0$.
\begin{defn}
    Let $V$ be a CVSA with $N$ odd. An element $W\in V^2$ is called a \textit{superpotential} of $V$ if $(W_{(\bf 0)}W)_{(\bf 0)} = 0$.
\end{defn}

Recall the $\beta\gamma$-system $V_{\beta\gamma}$ such that $\beta$ has cohomological degree $r$ and spin $\lambda$; and $\gamma$ has cohomological degree $N-r$ and spin $1-\lambda$. We assign a super grading to $V_{\beta\gamma}$ by requiring that $\beta$ and $\gamma$ have super grading $(-1)^{r}$ and $(-1)^{N-r+1}$ respectively. Thus $\beta$ is always bosonic and $\gamma$ is always fermionic. We denote the resulting CVSA as $FC^{(r)}_\lambda$.
% \begin{defn}
%     An element $W\in V^{N+1}$ is called a superpotential if $W_{(0)}W_{(-1)}\mathds 1$ belongs to the image of 
% \end{defn}

Let $\frak g$ be a finite-dimensional Lie algebra, and let $\{f^a_{bc}\}_{a,b,c}$ be the structure constants of $\frak g$ with respect to a basis $\{\mu_a\}_{a=1}^n$. Let $V$ be a CVA with a \textit{Hamiltonian $\frak g$ symmetry at level 0}; that is, $V$ comes with a morphism $V[\frak g]\to V$. We may form the tensor product CVA $V\otimes (V_{\beta\gamma})^{\otimes \dim\frak g}$. We denote the generating fields of $(V_{\beta\gamma})^{\otimes \dim\frak g}$ as $\{\beta^a,\gamma_a\}_{a=1}^{\dim \frak g}$. For now we suppose the gradings of these fields are arbitrary, and we will provide restrictions on them.

Now we summarize the setup. Let $N\in \Z_{\geq 1}$ be odd. We consider the tensor CVSA $V\otimes (FC_{N-1}^{(1)})^{\otimes \dim\frak g}$ such that $\beta_a$ is bosonic of cohomological degree $N-1$ and spin $\bf 1\in \N^n$; and $\gamma^a$ is fermionic of cohomological degree $1$ and spin $0$. Then we have the superpotential given by
\[W_{\frak g} = \tfrac{1}{2}f^a_{bc}\normord{\beta_a\gamma^b\gamma^c}  - \normord{\gamma^a\mu_a}\]
because $[W_{\frak g}(z),W_{\frak g}(w)] = 0$.
Here we are using Einstein summation notation.
We denote $D_{W_{\frak g}} = W_{(\bf 0)}$. Note that $[W_{\frak g}(z),W_{\frak g}(w)] \neq 0$ if $N$ is even. One may check that
\[D_{W_{\frak g}}\gamma^a(\bf z) = \frac{1}{2}f^a_{bc}\normord{\gamma^b\gamma^c}\!(\bf z),\quad D_{W_{\frak g}}\beta_a(\bf w) = -\mu_a(\bf z).\]
\[D_{W_{\frak g}}O(\bf z) = \normord{c^a(\mu_{a,(-1)}O)(\bf z)},\]
where $O\in V$. Note that $W_{(0)}$ is a spin 0 operator as well. We define $V/\!\!/\frak g$ to be the dg CVSA $V\otimes (FC_{N-1}^{(1)})^{\otimes \dim\frak g}$ together with the superpotential $W_{\frak g}$.

Here are two ways we can generalize the above construction, again following the work in \cite{ravioli}. If $V$ already comes equipped with another superpotential $W$, then we can check that $W+W_{\frak g}$ is a superpotential if $[W(\bf z),\mu_a(\bf w)] = 0$ for all $a=1,\ldots,\dim \frak g$. Suppose we are given a collection of symmetric, $\frak g$-invariant bilinear forms $\{K_{ab}^i\}_{i=1}^{\dim \frak g}$ on $\frak g$. Then we can further define
\[W_{\frak g,K} = \tfrac{1}{2}f^a_{bc}\normord{\beta_a\gamma^b\gamma^c} + \tfrac{1}{2}K^i_{ab}\normord{\gamma^a\partial_{z_i}\gamma^b}.\]
Then we have that
\begin{align*}
    [W_{\frak g,K}(\bf z),W_{\frak g,K}(\bf w)] &=\Delta(\bf z,\bf w)\left(\tfrac{1}{4}K_{ad}f_{bc}^d\normord{\gamma^a\gamma^b\partial_{w_i}\gamma^c}(\bf w)\right)\\
    &= \Delta(\bf z,\bf w)\left(\tfrac{1}{12}K_{ad}f_{bc}^d\partial_{w_i}\normord{\gamma^a\gamma^b\gamma^c}(\bf w)\right).
\end{align*}
This implies that $[W_{\frak g,K,(\bf 0)},W_{\frak g,K,(\bf 0)}] = 0$, so we have the following result:
\begin{thm}
    Let $V$ be a CVA with a morphism $V[\frak g]\to V$, where $V[\frak g]$ is generated by the fields $\{\mu_a(z)\}_{a=1}^{\dim \frak g}$. Choose a collection of symmetric, $\frak g$-invariant bilinear forms on $\frak g$ and $W$ a $\frak g$-invariant superpotential. Then
    \[W_{tot} = \tfrac{1}{2}f^a_{bc}\normord{\beta_a\gamma^b\gamma^c} + \tfrac{1}{2}K^i_{ab}\normord{\gamma^c\partial_{z_i}} - \normord{\gamma^c\mu_a} + W\]
    is a superpotential on $V/\!\!/\frak g = V\otimes (FC^{(1)}_{N-1})^{\otimes\frak g}$. If $V$ has a collection of conformal elements $\{\Gamma^i(\bf z)\}_{i=1}^n$ such that $W$ is a primary, then $V/\!\!/\frak g$ has a collection of conformal elements $\{\Gamma^i_{tot}(\bf z)\}_{i=1}^n$ given by
    \[\Gamma^i_{tot} = \Gamma - \normord{\beta_a\partial_{z_i}\gamma^a}.\]
\end{thm}
In the above setting, we take cohomology with respect to the differential on $V/\!\!/\frak g$, giving the \textit{cohomological W-algebra} $\cal W(\frak g;V) = H^\bullet(V/\!\!/\frak g,D_{W_{\frak g}})$, a conformal CVSA. If $V = V[\frak g]$, we write $\cal W(\frak g) \eqdef \cal W(\frak g;V[\frak g])$. The above constructions require $K^\ell = 0$ for all $\ell$ in the affine CVA, so in analogy to VOAs the level must be $k=0$.

% We require that
% \begin{align*}
%     (-1)^{p(\beta)(p(\gamma)+N) + N+1} &= 1\\
%     (-1)^{Np(\gamma)+p(\gamma)+p(\beta) + N+1} &= 1.
% \end{align*}
% For simplicity, suppose now that $N\in \Z_{\geq 1}$. Note that the cohomological degrees of $\beta$ and $\gamma$ add up to $N$. For simplicity, suppose we are not working with a super grading for now. Then the above equations read as
% \begin{align*}
%     (-1)^{p(\beta) + N+1} &= 1\\
%     (-1)^{Np(\gamma)+1} &= 1.
% \end{align*}
% If $N$ is odd, then the first and second equations require that $p(\beta)$ and $p(\gamma)$ be even and odd respectively.

% \subsection{A W-algebra example}
% We anticipate that CVAs naturally admit some analog theory of W-algebras, though we are unsure of all of the details. For now, we give a natural generalization of one of the most well-known W-algebras.

% Let $\cal W_k$ be the CVA freely generated by the fields $J(z),T(z)$, and $G^\pm(z)$ satisfying the relations
% \begin{align*}
%     [J(z),J(w)] &= [G^{\pm}(z),G^{\pm}(w)] = 0,\\
%     [J(z),G^\pm(w)] &= \pm\Delta(\bf z,\bf w)G^\pm(w),\\
%     [T(z),T(w)] &= 2\partial_w\Delta(\bf z,\bf w)T(w) + \Delta(\bf z,\bf w)\partial_wT(w)\\
%     [T(z),G^\pm(w)] &= \frac{3}{2}\partial_w\Delta(\bf z,\bf w)G^\pm(w) + \Delta(\bf z,\bf w)\partial_wG^\pm(w)\\
%     [T(z),J(w)] &= \partial_w\Delta(\bf z,\bf w)J(w) + \Delta(\bf z,\bf w)\partial_wJ(w)\\
%     [G^+(z),G^-(w)] &= -\frac{3}{2}\Delta(\bf z,\bf w)J(w)
% \end{align*}

\bibliographystyle{amsalpha}
\bibliography{bibfile}

\end{document}